\documentclass{article}

\usepackage{geometry}
\usepackage{color}
\usepackage{slashed}

\usepackage[numbers,sort&compress]{natbib}
\usepackage{graphicx,latexsym,euscript,makeidx,color,bm}
\usepackage{amsmath,amsfonts,amssymb,amsthm,thmtools,mathtools,mathrsfs,enumerate}
\usepackage[
hypertexnames=false,
colorlinks,
linkcolor=blue,
anchorcolor=green,
citecolor=red
]{hyperref}

\def\5n{\negthinspace \negthinspace \negthinspace \negthinspace \negthinspace }
\def\4n{\negthinspace \negthinspace \negthinspace \negthinspace }
\def\3n{\negthinspace \negthinspace \negthinspace }
\def\2n{\negthinspace \negthinspace }
\def\1n{\negthinspace }

\def\dbE{\mathbb{E}}     
\def\dbF{\mathbb{F}} \def\sF{\mathscr{F}}    
         
\def\dbH{\mathbb{H}}

\def\dbN{\mathbb{N}}     
     
\def\dbP{\mathbb{P}}     
\def\dbQ{\mathbb{Q}}     
\def\dbR{\mathbb{R}}

\def\Om{\Omega}

\def\ms{\medskip}

\def\no{\noindent}        \def\q{\quad}                      
                        
    \def\hb{\hbox}

            \def\({\Big (}
                  \def\){\Big )}
\def\leq{\leqslant}       \def\geq{\geqslant}
       \def\esssup{\mathop{\rm esssup}}   \def\[{\Big[}
        \def\essinf{\mathop{\rm essinf}}   \def\]{\Big]}
                   \def\cd{\cdot}

\def\e{\varepsilon}             
           \def\i{\infty}   

\theoremstyle{plain}

\theoremstyle{rmk}

\makeatletter

\@addtoreset{equation}{section}
\makeatother

\newtheorem{theorem}{Theorem}[section]

\newtheorem{proposition}[theorem]{Proposition}
\newtheorem{corollary}[theorem]{Corollary}
\newtheorem{lemma}[theorem]{Lemma}
\newtheorem{remark}[theorem]{Remark}

\newtheorem{assumption}{Assumption}

\makeatletter

\@addtoreset{equation}{section}
\makeatother

\allowdisplaybreaks[4]

\def\sqr#1#2{{\vcenter{\vbox{\hrule height.#2pt
				\hbox{\vrule width.#2pt height#1pt \kern#1pt \vrule width.#2pt}
				\hrule height.#2pt}}}}

\begin{document}
	\title{\Large \bf A global stochastic maximum principle for forward-backward stochastic control systems with quadratic  convex generator and unbounded terminal condition\thanks{YH is partially supported by Lebesgue Center of Mathematics ``Investissements d'avenir'' program-ANR-11-LABX-0020-01.
		Jiaqiang Wen is supported by National Natural Science Foundation of China (Grant No. 12571478) and Guangdong Basic and Applied Basic Research Foundation (Grant No. 2025B1515020091), and Shenzhen Science and Technology Program (Grant No. JCYJ20230807093309021).
	}
	}
	\author{
	Ying Hu\thanks{Univ. Rennes, CNRS, IRMAR - UMR 6625, F-35000 Rennes, France
		(Email: {\tt ying.hu@univ-rennes1.fr}).}~,~~~
	Feng Li\thanks{
			Corresponding author. 
		Department of Mathematics,
		Southern University of Science and Technology, Shenzhen, Guangdong, 518055, China
		(Email: {\tt 12331011@mail.sustech.edu.cn}).}~,~~~
	Jiaqiang Wen\thanks{ Department of Mathematics and SUSTech International Center for Mathematics,
		Southern University of Science and Technology, Shenzhen, Guangdong, 518055, China
		(Email: {\tt wenjq@sustech.edu.cn}).}
	}
	\date{}
	\maketitle
	\no\bf Abstract. \rm
We establish a global stochastic maximum principle for forward-backward stochastic control systems with an unbounded terminal condition and a generator that has quadratic growth and is convex in $Z$. The control domain need not be convex, and the diffusion coefficient may depend on the control. In this setting, the BMO martingale estimates used for quadratic BSDEs with bounded terminal conditions are no longer directly available. To overcome this difficulty, we work under an equivalent probability measure associated with the optimal trajectory and establish the estimates needed for the variational analysis. In particular, using a representation formula for convex quadratic BSDEs with conditional estimates, we derive estimates for the backward-state perturbations under this fixed measure. These estimates allow us to justify a second-order expansion of the cost functional and obtain the global stochastic maximum principle.

	\ms
	
	\no\bf Key words: 
	\rm 
forward-backward stochastic control systems, unbounded terminal condition, nonconvex control domain, an equivalent probability measure, stochastic maximum principle.
	\ms 
	
	\no\bf AMS subject classifications. \rm 93E20, 60H20, 49K45.
	\section{Introduction}

	Backward stochastic differential equations (BSDEs) provide a natural
	framework for stochastic control problems in which the performance
	criterion is defined recursively. When the generator has quadratic
	growth in $Z$, this framework also accommodates
	important nonlinear criteria, including those arising in risk-sensitive
	control. The combination of quadratic growth and an unbounded terminal
	condition, however, creates substantial difficulties for variational
	analysis. These difficulties become particularly pronounced when the
	control domain is nonconvex and the control enters the diffusion
	coefficient. The purpose of this paper is to establish a global
	stochastic maximum principle in this setting.
	
	The theory of BSDEs originated with the well-posedness result of
	Pardoux and Peng  \cite{Peng1990} for Lipschitz generators. Subsequent developments
	extended this theory to generators with quadratic growth in $Z$.
	For bounded terminal conditions, Kobylanski \cite{kobylanski2000backward} established
	fundamental existence and uniqueness results. BMO martingale methods,
	used in the quadratic BSDE setting by Hu, Imkeller, and M$\ddot{\rm u}$ller \cite{hu2005utility},
	became an important tool in the analysis of such equations.
	Tevzadze \cite{Tevzadze2008quadratic} combined these methods with a contraction argument,
	while further developments were obtained by Hu and Tang \cite{Hu-Tang-16} and Fan, Hu, and Tang \cite{fan2023multi}.
	
	Unbounded terminal conditions require a different set of techniques.
	Briand and Hu \cite{Briand2006quadratic, Briand and Hu} developed existence and uniqueness results using
	exponential integrability and, for uniqueness under convexity,
	the $\theta$-difference method. Delbaen, Hu, and Richou \cite {delbaen2011uniqueness} studied
	quadratic BSDEs with convex generators through 
	the Fenchel transform. Fan, Hu, and Tang \cite{Fan S} studied quadratic BSDEs
	under a strictly quadratic growth condition on the generator
	with respect to $Z$. Under their assumptions, strictly quadratic growth yields exponential integrability of
	$\int_0^T |Z(t)|^2\,dt$.  Other results concerning
	unbounded terminal conditions can be found in \cite{hu2022quadratic, delbaen2015uniqueness, Briand-Karoui-13, richou2012markovian}.
	These works provide the analytical foundation for the control
	problem considered here.
	
	The derivation of maximum principles for forward-backward stochastic
	control systems has developed alongside the theory of BSDEs.
	Following the open problem posed by Peng \cite{peng1999open}, 
	Hu \cite{Hu} subsequently developed a direct variational approach based
	on suitable first- and second-order adjoint equations.
	For quadratic BSDEs with bounded terminal conditions,
	Hu, Ji, and Xu \cite{hu2022global} established a stochastic maximum principle
	using BMO martingale estimates. Further contributions to
	forward-backward stochastic control can be found in
	\cite{buckdahn2024global, hao2020global, hu2018global, lv2016maximum, wu2013general,   hu2020maximum, huang2012maximum, Shi2006, li2026global}.

		We now describe our control problem. Fix a finite time horizon
		$T>0$ and a complete probability space
		$(\Omega,\mathscr{F},\mathbb P)$ carrying a $d$-dimensional
		standard Brownian motion $W$.
		Let $\mathbb F=(\mathscr{F}_t)_{t\geq0}$ be the natural filtration
		of $W$, augmented by all $\mathbb P$-null sets in $\mathscr{F}$.
		Given a deterministic initial state $x(0)\in\mathbb R^n$ and
		a nonempty set $U\subseteq\mathbb R^k$, consider
		\begin{equation}\label{system}
			\left\{
			\begin{aligned}
				dX(t)
				&=b(t,X(t),u(t))\,dt
				+\sigma(t,u(t))\,dW(t),\\
				dY(t)
				&=f(t,X(t),Y(t),Z(t),u(t))\,dt
				-Z(t)^\top dW(t),\qquad t\in[0,T],\\
				X(0)&=x(0),\qquad
				Y(T)=h(X(T)).
			\end{aligned}
			\right.
		\end{equation}
		Here,
	 \begin{equation*}
		\begin{aligned}
			b&:\Omega\times[0,T]\times\mathbb R^n\times U
			\longrightarrow\mathbb R^n,\\
			\sigma&:\Omega\times[0,T]\times U
			\longrightarrow\mathbb R^{n\times d},\\
			f&:\Omega\times[0,T]\times\mathbb R^n
			\times\mathbb R\times\mathbb R^d\times U
			\longrightarrow\mathbb R.
		\end{aligned}
	\end{equation*}
		Writing $\mathcal P_{\mathrm{prog}}$ for the progressive
		$\sigma$-algebra associated with $\mathbb F$, we assume that
		$b$, $\sigma$, and $f$ are jointly measurable with respect to
		$\mathcal P_{\mathrm{prog}}$ and the Borel $\sigma$-algebras
		of their remaining arguments.
		The terminal function
		$h:\Omega\times\mathbb R^n\to\mathbb R$
		is $\mathscr{F}_T\otimes\mathcal B(\mathbb R^n)$-measurable.
		The generator has quadratic growth in $Z$, and the terminal
		value is allowed to be unbounded.
		The admissible controls are given by
		\begin{equation}
			\mathcal U_{\mathrm{ad}}
			=
			\left\{
			u:\Omega\times[0,T]\to U:
			u\text{ is }\mathbb F\text{-progressively measurable}
			\right\}.
		\end{equation}
		For each admissible control, the associated cost is
		$J(u(\cdot))=Y(0)$.
		Our objective is to find $u^*(\cdot)\in\mathcal U_{\mathrm{ad}}$
		such that
		\begin{equation}
			J(u^*(\cdot))
			=
			\inf_{u(\cdot)\in\mathcal U_{\mathrm{ad}}}J(u(\cdot)).
		\end{equation}
		We refer to this optimization problem as {\bf Problem  (O)}. The set $U$ is not required to be convex, and the diffusion
		coefficient may depend on the control.

		The principal analytical obstacle is the loss of the BMO
		estimates available in the bounded-terminal setting.
		When the terminal condition is unbounded, the BMO property of
		$ Z\cd W$ is no longer guaranteed.
		Consequently, the BMO-based arguments used in Hu, Ji, and Xu \cite{hu2022global} to handle
		the coefficient $f_z$ and obtain estimates for the adjoint
		and variational equations cannot be applied directly. By a result of Delbaen, Hu, and Richou \cite{delbaen2011uniqueness}, we know that $\mathscr{E}\big(f_z\cd W\big)$ belongs to the space $L\log L$  (i.e., $\dbE\big[\mathscr{E}\big(f_z\cd W\big)_T\int_{0}^{T}|f_z(s)|^2ds\big] < \i$).  Using the exponential moment assumptions on the terminal
		condition and an $L\log L$ estimate for the localized
		stochastic exponentials, Delbaen, Hu, and Richou
		\cite{delbaen2011uniqueness} proved that
		$\mathscr{E}\bigl( f_z\cd W\big)$
		is a uniformly integrable martingale. Inspired by this,  we introduce a probability measure,
		equivalent to $\mathbb P$
		and associated with the optimal trajectory.
		Under $\mathbb Q^*$, Girsanov's transformation eliminates
		the terms involving $f_z$ from the generators of the
		relevant linear BSDEs with random coefficients.
		Two further difficulties remain.
		First, the nonconvexity of the control domain, together
		with the control dependence of the diffusion coefficient,
		requires second-order variational equations and
		higher-order moment estimates for both the forward-
		and backward-state differences.
		Second, even under the new measure, estimating the
		backward-state differences remains challenging.
		After introducing another probability measure,
		$\mathbb Q^\varepsilon$, associated with the perturbed
		trajectory, the representation formula established by
		Delbaen, Hu, and Richou \cite{delbaen2011uniqueness}
		allows us to analyze auxiliary BSDEs with generators
		independent of $Z$.
		However, these estimates involve conditional expectations
		under both $\mathbb Q^*$ and $\mathbb Q^\varepsilon$,
		whereas the desired estimates must be established
		under the fixed measure $\mathbb Q^*$.

			To address the difficulties arising from the nonconvex
		control domain, we first extend the integrability result
		for $Y$ in quadratic BSDEs with unbounded terminal
		conditions established by Briand and Hu
		\cite{Briand and Hu}.
		More precisely, stronger integrability of the terminal
		condition and the inhomogeneous term yields
		stronger integrability of $Y$ (see \autoref{system-wellposedness}).
		This enables us to prove \autoref{q*}, which in turn
		allows us to handle the terms arising in the second-order
		expansion required by the nonconvexity of the control domain. 
		To estimate the backward-state differences
		in \autoref{importan-}, we proceed as follows.
		First, we use the representation result of Delbaen, Hu,
		and Richou \cite[line 14, page 567]{delbaen2011uniqueness}
		and Gronwall's inequality to derive conditional estimates
		for the difference between the solutions of two auxiliary
		BSDEs whose generators have properties that allow us
		to derive \eqref{g}. 
		The right-hand side involves conditional expectations
		of the squared suprema of the forward-state difference
		and the optimal forward state over $[t,T]$.
		Second, estimates for these forward processes reduce
		the problem to controlling 
		$\dbE^{\dbQ^\e}_t\big[\int_t^T |q(s)|^2ds\big]$ and $\dbE^{\dbQ^*}_t\big[\int_t^T |q(s)|^2ds\big]$, where $q$ is the Girsanov kernel
		associated with either the optimal or the perturbed
		trajectory, under the corresponding probability measure.
		Third, to control the terms arising in the previous step,
		we need an estimate for
		$\dbE^{\mathbb Q^*}\big[\int_0^T |Y^*(s)|^2\,ds\big]$, where $Y^*$ denotes the backward-state process of the
		original control system corresponding to the optimal
		control $u^*$.
		However, this term cannot be estimated using the auxiliary
		BSDE, since the conditions on its generator do not yield
		an estimate as favorable as \eqref{g}.
		We therefore need to return to the original equation
		in \eqref{system}. A direct bound in terms of exponential
		moments of the forward states would introduce nested
		exponential terms in the subsequent analysis, beyond
		the available integrability.
		We therefore use the logarithmic estimate of Briand
		and Hu \cite[Proposition 1]{Briand and Hu}, together
		with conditional exponential estimates for the
		forward-state increments.
        This resolves the difficulty in estimating the difference
        between the perturbed and optimal $Y$-components.
        Finally,
		we use the convexity of the generator in $Z$ 
		to estimate the difference between the corresponding
		martingale integrands, completing the estimates for
		the backward-state differences.

		Our framework is related to several existing control models.
		For generators of the form
		$
		f(t,x,y,z,u)
		=
		\frac{\gamma}{2}|z|^2+g(t,x,u),
		$
		a suitable exponential transformation connects Problem~(O)
		with risk-sensitive control, as studied by Lim and Zhou \cite{lim2005new}
		and Whittle \cite{whittle1990risk}. In the setting of Hu, Ji, Xu, and Xue
		\cite[Section 3.3]{hu2023bsde}, the generator is
		independent of $Y$ and the diffusion is independent of
		the control, so the forward spike-variation analysis
		does not require the same second-order expansion.
		Here, both a $Y$-dependent generator and a
		control-dependent diffusion coefficient are allowed.
		Relative to the framework of Delarue and Lavigne \cite{delarue2026robust},
		we allow a nonconvex control domain and do not require
		the drift and diffusion coefficients to be linear. In \autoref{section7}, we also establish a stochastic maximum principle for the convex control case using our approach, under weaker assumptions on the diffusion coefficient than those required in the nonconvex case (see \autoref{cor}). 
		 
		The contributions of this paper can therefore be summarized
		as follows. First, we establish a global stochastic maximum
		principle for a class of forward-backward control systems
		with unbounded terminal conditions and generators of
		quadratic growth that satisfy a convexity condition
		in $Z$. The control domain may be nonconvex, and the control
		may enter the diffusion coefficient.
		Second, we develop the estimates for the forward variations,
		adjoint equations, and backward variational equations under
		the measure $\mathbb Q^*$, without relying on the BMO property
		used in the bounded-terminal setting.
		Third, we obtain the backward perturbation estimate
		in \autoref{importan-} through the representation formula established by Delbaen, Hu, and Richou \cite{delbaen2011uniqueness}, which plays a crucial role in the proofs of \autoref{prop3.9} and \autoref{last}.

		The remainder of the paper is organized as follows.
		\autoref{section2} introduces the notation and assumptions and
		collects the preliminary well-posedness and integrability
		results.
		\autoref{section3} defines $\mathbb Q^*$ and establishes estimates
		for the forward-state perturbations and their
		first- and second-order approximations.
		\autoref{section4} derives the estimates for the backward-state
		perturbations.
		\autoref{section5} studies the adjoint equations and the backward
		variational equation.
		\autoref{section8} establishes the asymptotic expansion of the cost
		and controls the remainder terms.
		Finally, \autoref{section7} establishes the stochastic maximum principle for {\bf Problem (O)} and discusses the special case of a convex control domain under weaker assumptions.

	\section{Preliminaries}\label{section2}
    The notation $\mathbb{R}^{m \times d}$ denotes the space of $m \times d$-matrix $M$, equipped with the Euclidean norm defined as $|M| = \sqrt{\text{tr}(MM^{\top})}$, where $M^{\top}$ is the transpose of $M$.
    %
    For $\beta \geq 1$, a Euclidean space $\dbH$, and a probability measure $\dbQ$, we introduce the following spaces:
    \begin{align*}
    	&S^\beta_{\dbF}(\dbQ)
    	= \Big\{Y:\Om\times [0,T] \to\dbH\Bigm| Y~ \hb{is $\dbF$-progressively measurable, continuous, and~} \\
    	& \quad \quad \quad \quad \quad \quad\quad\ \quad \quad \quad   \|Y\|_{S^\beta_{\dbF}(\dbQ)} := \Big(\dbE^{\dbQ}[\mathop{\sup}\limits_{t \in [0,T]} |Y_{t}|^\beta]\Big)^\frac{1}{\beta}<\i \Big\}, \\
    	&S^\i_{\dbF}(\dbQ)
    	= \Big\{Y:\Om\times [0,T] \to\dbH\Bigm| Y~ \hb{is $\dbF$-progressively measurable, continuous,
    		and~} \\
    	& \quad \quad \quad \quad \quad \quad\quad\ \quad \quad \quad   \|Y\|_{S^\i_{\dbF}(\dbQ)} := \hb{esssup} _{(s,\omega) \in [0,T]\times\Om}^{\dbQ} |Y_{s}(\omega)| 
    	<\i \Big\},\\
    	&L^\beta_{\dbF}(\dbQ)
    	= \Big\{Y:\Om\times [0,T] \to\dbH\Bigm| Y~ \hb{is $\dbF$-progressively measurable, 
    		and~} \\
    	& \quad \quad \quad \quad \quad \quad\quad\ \quad \quad \quad   \|Y\|_{L^\beta_{\dbF}(\dbQ)} \triangleq \bigg(\dbE^\dbQ\bigg[\Big(\int_{0}^{T}|Y_s|^2 ds\Big)^{\frac{\beta}{2}}\bigg]\bigg)^\frac{1}{\beta}<\i \Big\},\\
    	& \mathcal{E}(\dbP)
    	=\Big\{Y : \Om\times [0,T] \to\dbH \Bigm| \exp(|Y|)\in \mathop{\bigcap}\limits_{\beta \geq 1} S^\beta_{\dbF}(\dbP)   \Big\}, \\
    	%
    	&\mathcal{M}(\dbP)
    	=\Big\{ Z : \Om\times [0,T] \to\dbH\Bigm| Z\in \mathop{\bigcap}\limits_{\beta \geq 1}	L^\beta_{\dbF}(\dbP)\Big\}.    \end{align*}
    %
    %
    %
    %
    Denote by $\mathscr{E}(M)_t$ the Dol\'{e}ans-Dade exponential of a continuous local martingale $M$ with $M(0)=0$, that is, $\mathscr{E}(M)_t \triangleq \exp \{M_t - \frac{1}{2} \langle M\rangle_t\}$, for any $t\in [0, T]$. 
    \ms

Here are some assumptions concerning system \eqref{system}.
\begin{assumption}\rm\label{A1}
	\begin{itemize}
		\item [$\rm(i)$] 
		For $\mathbb P$-almost every $\omega$,
		$h(\cdot)$ is twice continuously differentiable.
		The derivatives $h_x$ and $h_{xx}$ are essentially
		bounded in $\omega$, uniformly in $x$ and $h(0)$ is essentially bounded.
		\item [$\rm(ii)$]For $d\mathbb P\otimes dt$-almost every $(\omega,t)$
		and every $u\in U$, $b(t,\cdot,u)$ is twice
		continuously differentiable.
		The derivatives $b_x$ and $b_{xx}$ are essentially
		bounded in $(\omega,t)$, uniformly in $(x,u)$.
		Moreover, there exists a deterministic constant $L_1>0$
		such that, for $d\mathbb P\otimes dt$-almost every
		$(\omega,t)$ and all $u\in U$,
		\begin{equation}\label{b}
			|b(t,0,u)|+|\sigma(t,u)|\le L_1.
		\end{equation}
		\item [$\rm(iii)$] 
		For $d\mathbb P\otimes dt$-almost every $(\omega,t)$
		and every $u\in U$, the mapping
		$(x,y,z)\mapsto f(t,x,y,z,u)$ is twice
		continuously differentiable and the mapping $z \mapsto f(t, x, y, z, u)$ is convex. 
		The derivatives $f_x,f_y,f_{xx},f_{xy},f_{xz},
		f_{yy},f_{yz},f_{zz}$ are essentially bounded in
		$(\omega,t)$, uniformly in $(x,y,z,u)$.
		There exist deterministic positive constants
		$L_2,L_3,L_4$
		such that, for $d\mathbb P\otimes dt$-almost every
		$(\omega,t)$ and all
		$x\in\mathbb R^n$, $y\in\mathbb R$,
		$z,z'\in\mathbb R^d$, and $u,u_1,u_2\in U$,
		\begin{align}
			&|f(t, 0, 0, 0, u)| \leq L_2;\label{condition}\\
			&|f_z(t, x, y, z, u)| \leq L_3 (1+ |z|);\\
			&|f(t, x, y, z, u_1) - f(t, x, y, z, u_2)| \leq L_4(1 + |x| ).\label{2.4}
		\end{align}
	\end{itemize}
\end{assumption}
\begin{remark}
	It is easy to check that $f(t, x, y, z, u) = (\sin x)\cd\sin (u^2) +y+z^2+z+\sin (u^2)$ with $n=d=k=1$ satisfies (iii) of \autoref{A1}.
\end{remark}
%
%

%
Now we give the well-posedness result of system \eqref{system} under \autoref{A1}. 
\begin{lemma} \label{system-wellposedness}
Let \autoref{A1} hold. Then, for $u \in \mathcal{U}_{ad}$,  the SDE in \eqref{system} admits a unique solution $X^u$  such that for $\beta > 0$ and $0 < \epsilon_0 < \frac{1}{6e^{2\|b_x\|_\i T}\|\sigma\|_\i^2 T} $,
\begin{align}\label{well-posed x}
\dbE\big[\mathop{\sup}_{0 \leq t \leq T} \exp\{\beta |X^u(t)|\}\big] \leq C_1 < \i,
\q \dbE\big[\mathop{\sup}_{0 \leq t \leq T} \exp\{\epsilon_0 |X^u(t)|^2\}\big] \leq C_2 < \i,
\end{align}
and the BSDE in \eqref{system} admits a unique solution $(Y^u, Z^u) \in \mathcal{E}(\dbP) \times \mathcal{M}(\dbP)$ and for any $0 < \epsilon_1 < \frac{1}{96\|\sigma\|^2_\i  Te^{2\|f_y\|_\i T + 2\|b_x\|_\i T}(\|h_x\|^2_\i + \|f_x\|^2_\i T^2)}$, 
\begin{align}\label{small}
\dbE\big[\mathop{\sup}_{0 \leq t \leq T} \exp\{\epsilon_1 |Y^u(t)|^2\}\big] \leq C_3 < \i.
\end{align}
\end{lemma}
\begin{proof}
On the one hand, by Briand and Hu \cite [page 563]{Briand and Hu}, we get the first part of \eqref{well-posed x}. Thanks to an argument similar to that in Briand and Hu \cite [page 563]{Briand and Hu}, we deduce that
\begin{align*}
&\dbE\Big[\mathop{\sup}_{t\in[0, T]}\exp\{\epsilon_0 |X^u(t)|^2\}\Big] \leq K \dbE\Big[\mathop{\sup}_{t\in[0, T]}\exp\Big\{3e^{2\|b_x\|_\i T} \epsilon_0 \Big|\int_{0}^{t}\sigma\big(s, u(s)\big)dW(s)\Big|^2\Big\}\Big]\\
 &\leq  K \dbE\Big[\mathop{\sup}_{t\in[0, \|\sigma\|^2_\i T]}\exp\Big\{3e^{2\|b_x\|_\i T} \epsilon_0 |W(t)|^2\Big\}\Big]  < \i,
\end{align*}
 which implies the second part of \eqref{well-posed x}.
On the other hand, note that the BSDE in \eqref{system} can be rewritten as 
\begin{align*}
- Y^u(t) = -h\big(X^u(T)\big)  + \int_{t}^{T} f\big(s, X^u(s), -\big(-Y^u(s)\big), Z^u(s), u(s)\big)ds - \int_{t}^{T} Z^u(s)^\top dW(s)
\end{align*}
and
it is easy to check that
\begin{align*}
|f(t, x, -y, z, u)| \leq \|f_y\|_\i |-y| + 2L_3 |z|^2 + L_2+L_3 + \|f_x\|_\i |x|.
\end{align*}
Then, thanks to \autoref{A1} and \eqref{well-posed x}, we deduce that for $\beta' > 1$,
\begin{align*}
&\dbE\Big[\exp\Big\{ \beta' \Big(|h(X^u(T))| + \int_{0}^{T}\{L_2 +L_3+ \|f_x\|_\i |X^u(s)|\}ds\Big)\Big\}\Big] \\
&\leq \dbE\Big[\exp\Big\{ \beta' \Big(|h(0)| + \|h_x\|_\i|X^u(T)| + \int_{0}^{T}\big\{L_2+L_3  + \|f_x\|_\i |X^u(t)|\big\}dt\Big)\Big\}\Big]< \i.
\end{align*}
It then follows from \autoref{A1} and Briand and Hu \cite[Corollary 6]{Briand and Hu} that the BSDE in \eqref{system} admits a unique solution $(Y^u, Z^u) \in \mathcal{E}(\dbP) \times \mathcal{M}(\dbP)$  and the following estimate holds: for $\dbP$-a.s.,
\begin{align*}
\forall t \in [0, T], \q |Y^u(t)| \leq  \frac{1}{4L_3}\log \dbE_t\Big[\exp \Big\{4L_3 e^{\|f_y\|_\i T}\Big(|h(X^u(T))| +  \int_{0}^{T}\{L_2 +L_3+ \|f_x\|_\i |X^u(s)|\}ds\Big)\Big\}\Big].
\end{align*}
Let us set
\begin{align*}
N_1 := 4L_3 e^{\|f_y\|_\i T}\Big(|h(X^u(T))| +  \int_{0}^{T}\{L_2 +L_3+ \|f_x\|_\i |X^u(s)|\}ds\Big).
\end{align*}
Then, we deduce that
\begin{align}\label{mn}
&\dbE\Big[\mathop{\sup_{0 \leq s \leq T}}\exp\Big\{16L_3^2 \epsilon'_1 |Y^u(s)|^2\Big\} \Big]
\leq \dbE\Big[\mathop{\sup_{0 \leq t \leq T}}\exp\Big\{ \epsilon'_1\log^2 \dbE_t\big[\exp\{N_1\}\big] \Big\} \Big],
\end{align}
where $\epsilon'_1$ is a small constant and will be determined later.
Define 
$
\psi(x):= xe^{\epsilon'_1 \log^2 x}, x\geq1
$
and $A(t):= \dbE_t[e^{N_1}]\geq 1$.
Then, one can check that $\psi''(x)> 0, x\geq1$. 
It then follows from conditional Jensen's inequality that 
\begin{align}\label{nn}
\exp\Big\{ \epsilon'_1\log^2 \dbE_t\big[\exp\{N_1\}\big]\Big\}  = e^ {\epsilon'_1\log^2 A(t)}  \leq \psi(A(t))= \psi\big(\dbE_t[e^{N_1}]\big) \leq \dbE_t\big[\psi(e^{N_1})\big] = \dbE_t\big[e^{\epsilon'_1N_1^2 + N_1}\big].
\end{align}
Furthermore, in view of \eqref{mn} and \eqref{nn}, by Doob's inequality and the fact that $y \leq y^2+1, y>1$, we obtain that for $0 < \epsilon'_1 < \frac{1}{1536\|\sigma\|^2_\i L_3^2 Te^{2\|f_y\|_\i T + 2\|b_x\|_\i T}(\|h_x\|^2_\i + \|f_x\|^2_\i T^2)}$,
\begin{align*}
&\dbE\Big[\mathop{\sup_{0 \leq t \leq T}}\exp\Big\{16L_3^2 \epsilon'_1 |Y^u(t)|^2\Big\} \Big] \leq \dbE\Big[\mathop{\sup_{0 \leq t \leq T}}\dbE_t\big[e^{\epsilon'_1N_1^2 + N_1}\big] \Big] \leq \dbE\Big[\mathop{\sup_{0 \leq t \leq T}}\Big(\dbE_t\big[e^{\epsilon'_1N_1^2 + N_1}\big]\Big)^2  \Big] + 1\\
&
 \leq 4\dbE\big[e^{2\epsilon'_1N_1^2 + 2N_1}\big] +1 \leq 4\dbE\big[e^{4\epsilon'_1N_1^2}\big] ^{\frac{1}{2}}\cd\dbE\big[e^{4N_1}\big] ^{\frac{1}{2}}+1 < \i,
\end{align*}
which yields \eqref{small}.
\end{proof}

\begin{remark}
In view of \eqref{b} and \eqref{condition}, it is easy to check that $C_1, C_2, C_3$ are independent of the control term $u$.
\end{remark}
%

%
\section{A new probability measure and the first- and second-order variational equations for $X^\e-X^*$} \label{section3}
In this section, we introduce a new probability measure and the first- and second-order variational equations for $X^\e-X^*$ under this measure.
Throughout the remainder of the paper, $K$ denotes a generic positive
constant that may vary from line to line. For simplicity, we restrict
our analysis to the scalar case $n=d=k=1$; the general case can be
handled similarly.

Let $u^*(\cdot)$ be an optimal control and $(X^*(\cdot),Y^*(\cdot),Z^*(\cdot))$ be the corresponding state trajectories of \eqref{system}. 
Since the control domain $U$ is not necessarily convex, we will use the spike variation method. 
For any $u(\cdot) \in \mathcal{U}_{ad}$, $0 < \e< T-t_0$, and any fixed $t_0 \in [0, T)$, define $u^\e(\cd)$ as follows:
\begin{equation}\label{}u^\e(t)\triangleq \left\{\begin{aligned}
&u^*(t), \q t  \notin [t_0, t_0 + \e];\\
&u(t), \q t \in [t_0, t_0+\e],
\end{aligned}\right.\end{equation}
which is a perturbed admissible control.
In addition, we let $(X^\e(\cdot), Y^\e(\cdot), Z^\e(\cdot))$ denote the state trajectories of \eqref{system} associated with $u^\e (\cdot)$. 
For simplicity,
for $\phi = \{h, f, b, \sigma\}$ and $w = \{x, y,  z,  xx, xy, xz, yy, yz, zz\}$,  we denote
\begin{equation}
\begin{aligned}\label{notation}
&\phi(t) \triangleq \phi(t, X^*(t), Y^*(t), Z^*(t), u^*(t)), \q
\phi_{w}(t) \triangleq \phi_{w}(t, X^*(t), Y^*(t),  Z^*(t), u^*(t)),\\
&\delta \phi (t) \triangleq  \phi(t, X^*(t), Y^*(t),  Z^*(t),u(t)) - \phi(t), \q
\delta \phi_w (t)\triangleq \phi_w(t, X^*(t), Y^*(t), Z^*(t),u(t)) - \phi_w(t).
\end{aligned}\end{equation}

The first- and second-order variational equations for the SDE in \eqref{system} are
\begin{equation}\label{variational-x1}\left\{\begin{aligned}
dX_1(t) = &b_{x}(t)X_1(t)  dt +  \delta \sigma(t) 1_{[t_0, t_0 + \e]}(t)dW(t), \quad t\in [0,T],\\
X_1(0) =&0
\end{aligned}\right.\end{equation}
and
\begin{equation}\label{variational-x2}\left\{\begin{aligned}
dX_2(t) =& \big[b_{x}(t)X_2(t) + \delta b(t) 1_{[t_0, t_0 + \e]} (t)+ \frac{1}{2} b_{xx}(t) X_1(t) X_1(t) \big]dt, 
\quad t\in [0,T], \\
X_2(0) =&0,
\end{aligned}\right.\end{equation}
respectively.

\begin{remark}
Throughout this paper, whenever an estimate or an asymptotic relation involving the initial time $t_0$ of a spike variation is obtained by applying the Lebesgue differentiation theorem, it is understood to hold for Lebesgue-almost every $t_0\in[0,T)$, even if this is not explicitly stated. For each fixed admissible control $u(\cdot)$, the exceptional null set may depend on $u(\cdot)$.
\end{remark}

\subsection{A new probability measure and its properties}
In this subsection, we introduce a new probability measure and establish its properties.

Due to the lack of the BMO martingale technique, and
inspired by Delbaen, Hu, and Richou \cite {delbaen2011uniqueness}, we analyze it under a new measure.
Define 
\begin{align}\label{qd}
q^*(t) : = f_z(t), \q \frac{d{\dbQ^*}}{d\dbP} := \mathscr{E}\Big(\int_{0}^{\cd}q^*(s) dW(s)\Big)_T , \q dW^{q^*}(t) := dW(t) - q^*(t)dt.
\end{align}
Then,
we have the following result.
\begin{proposition}\label{prop3.2}
Let \autoref{A1} hold. Then, $\dbQ^*$ is a probability measure, $W^{q^*}$ is a Brownian motion under $\dbQ^*$, and $\dbE^{\dbQ^*}\Big[\int_{0}^{T}|q^*(t)|^2dt\Big] < \i$.
\end{proposition}
\begin{proof}
Let $f^*(t, y, z) = f\big(t, X^*(t), y, z, u^*(t)\big)$. Then, we have 
\begin{align*}
Y^{*}(t) = h\big(X^{*}(T)\big) - \int_{t}^{T} f^*\big(s, Y^{*}(s), Z^{*}(s)\big)ds + \int_{t}^{T}Z^{*}(s) dW(s).
\end{align*}
Thanks to the well-posedness of $X^*$ in \autoref{system-wellposedness}, one can verify that condition (3.1) in Delbaen, Hu, and Richou \cite{delbaen2011uniqueness}, specialized to the case $N=1$, is satisfied.
Then, in view of  the proof of Delbaen, Hu, and Richou \cite [Theorem 3.3 and Lemma 3.4]{delbaen2011uniqueness},
we obtain that 
$\Big\{\mathscr{E}\Big(\int_{0}^{\cd}f^*_z\big(s,  Y^{*}(s),  Z^{*}(s) \big) dW(s)\Big)_t\Big\}_{0\leq t \leq T} $
is a martingale. It is easy to check that 
$
f^*_z\big(s,  Y^{*}(s),  Z^{*}(s) \big)  = f_z(s).
$ 
It then follows from Girsanov theorem that $\dbQ^*$ is a probability measure and $W^{q^*}$ is a Brownian motion under $\dbQ^*$.
In addition, in view of equation (3.3) in Delbaen, Hu, and Richou \cite {delbaen2011uniqueness}, we also have $\dbE^{\dbQ^*}\Big[\int_{0}^{T}|q^*(s)|^2ds\Big] < \i$.
Therefore, we obtain the desired result.
\end{proof}
\begin{proposition}\label{ystate}
	Let \autoref{A1} hold. Then, we have
	\begin{align*}
		\dbE^{\dbQ^*}\Big[\mathop{\sup}_{s\in[0, T]}|Y^\e(s)|^2\Big] + \dbE^{\dbQ^*}\Big[\mathop{\sup}_{s\in[0, T]}|Y^*(s)|^2\Big]   \leq K_1 < \i,
	\end{align*}
	where $K_1$ is independent of $\e$.
\end{proposition}
\begin{proof}
	By the fact that $xy\leq e^{\beta x} + \frac{y}{\beta}(\log y - \log \beta -1), (x,y, \beta) \in \dbR\times (0, \i) \times (0, \i)$, and \autoref{system-wellposedness}, one can get the desired result. 
\end{proof}

The following result shows that $q^* \in L^4_{\dbF}(\dbQ^*)$ and $Z^* \in L^2_{\dbF}(\dbQ^*)$.
\begin{proposition}\label{q*}
Let \autoref{A1} hold. Then, we have
\begin{align}\label{qq}
\dbE^{\dbQ^*}\Big[\Big(\int_{0}^{T}|q^*(s)|^2ds\Big)^2\Big]  + \dbE^{\dbQ^*}\Big[\int_{0}^{T}|Z^*(s)|^2ds\Big]<\i.
\end{align}
\end{proposition}
\begin{proof}
Recall that $f^*(t, y, z) = f\big(t, X^*(t), y, z, u^*(t)\big)$ and
define
\begin{align*}
g^*(t, y, q) = \mathop{\sup}_{z\in \dbR^{d}}\Big(z q - f^*(t, y, z)\Big).
\end{align*}
One can check that
\begin{align*}
f(t, x, y, z, u) \leq L_2 + L_3 + \|f_x\|_\i |x| + \|f_y\|_\i |y| + 2L_3|z|^2,
\end{align*}
which implies that 
\begin{align}\label{ge}
g^*(t, y, q) \geq -\Big(L_2 + L_3 + \|f_x\|_\i |X^*(t)|\Big) - \|f_y\|_\i |y| + \frac{|q|^2}{8L_3}.
\end{align}
Let
\begin{align*}
&a(t):= L_2 + L_3 + \|f_x\|_\i |X^*(t)| + \|f_y\|_\i |Y^*(t)|, \q \tilde{A}(t):= \int_{0}^{t}\{g^*\big(s, Y^*(s), q^*(s) \big)+a(s)\}ds,\\
&\tilde{\tau}_n:= \inf\Big\{t: \tilde{A}(t)\geq n \q\hb{or} \q \int_{0}^{t}|Z^*(s)|^2ds \geq n\Big\} \wedge T.
\end{align*}
Then, one can check that 
\begin{align}\label{at}
\tilde{A}(t) \geq \frac{1}{8L_3}\int_{0}^{t}|q^*(s)|^2ds\geq 0.
\end{align}
Applying It$\hat{\rm o}$'s formula to $\big\{\tilde{A}(t)\{\tilde{A}(t)+\int_{0}^{t}Z^*(s)dW^{q^*}(s)\}\big\}_{t\in [0, \tilde{\tau}_n]}$, one can deduce that
\begin{align}\label{q4}
\frac{1}{2}\dbE^{\dbQ^*}\big[|\tilde{A}(\tilde{\tau}_n)|^2\big] = \dbE^{\dbQ^*}\Big[\tilde{A}(\tilde{\tau}_n)\Big\{\tilde{A}(\tilde{\tau}_n) + \int_{0}^{\tilde{\tau}_n}Z^*(s)dW^{q^*}(s)\Big\}  - \int_{0}^{\tilde{\tau}_n}\Big\{\tilde{A}(s)+\int_{0}^{s}Z^*(k)dW^{q^*}(k)\Big\}d\tilde{A}(s)\Big].
\end{align}
According to the proof of Delbaen, Hu, and Richou \cite [Theorem 3.3]{delbaen2011uniqueness}, we have that 
\begin{align}\label{zw}
Y^{*}(0) = Y^{*}(t) + \int_{0}^{t} g^*\big(s, Y^{*}(s), q^*(s)\big)ds + \int_{0}^{t}Z^{*}(s) dW^{q^*}(s).
\end{align}
It then follows from \autoref{state}, \autoref{ystate}, and \eqref{zw} that
\begin{align*}
\dbE^{\dbQ^*}\Big[\mathop{\sup}_{s\in [0, T]}\Big|\tilde{A}(s)+\int_{0}^{s}Z^*(k)dW^{q^*}(k)\Big|^2\Big] \leq K \dbE^{\dbQ^*}\Big[\mathop{\sup}_{s\in [0, T]}\Big\{|Y^*(s)|^2 + \int_{0}^{s}|a(k)|^2dk\Big\}\Big]<\i.
\end{align*}
Using H$\ddot{\rm o}$lder's inequality to \eqref{q4}, one can deduce that
\begin{align*}
\dbE^{\dbQ^*}\big[|\tilde{A}(\tilde{\tau}_n)|^2\big] \leq K\dbE^{\dbQ^*}\Big[\mathop{\sup}_{s\in [0, T]}\Big|\tilde{A}(s)+\int_{0}^{s}Z^*(k)dW^{q^*}(k)\Big|^2\Big] <\i.
\end{align*}
Thus, by letting $n\rightarrow \i$ and using Fatou's lemma, we obtain that
\begin{align*}
	\dbE^{\dbQ^*}\big[|\tilde{A}(T)|^2\big]
 <\i.
\end{align*}
It then follows from \eqref{at} that $\dbE^{\dbQ^*}\Big[\Big(\int_{0}^{T}|q^*(s)|^2ds\Big)^2\Big] <\i$.
Furthermore, thanks to \eqref{zw} and It$\hat{\rm o}$'s isometry, one can deduce that
\begin{align}\label{zd}
\dbE^{\dbQ^*}\Big[\int_{0}^{T}|Z^*(s)|^2ds\Big] \leq K \dbE^{\dbQ^*}\Big[|\tilde{A}(T)|^2+ \int_{0}^{T}|a(s)|^2ds + 2\mathop{\sup}_{s\in [0, T]}|Y^*(s)|^2\Big]  < \i.
\end{align}
This completes the proof.
\end{proof}

\subsection{The estimates for $X^\e - X^*$ under the probability measure $\dbQ^*$}
In this subsection, we give estimates for $X^\e - X^*$ under the probability measure $\dbQ^*$.
Recall that $X_1$ and $X_2$ are defined in \eqref{variational-x1} and \eqref{variational-x2}, respectively.
Set 
\begin{align}\label{hatx}
\hat{X}^\e(t) = X^\e(t) - X^*(t), \q \hat{X}^\e_1(t) := X^\e(t) - X^*(t) - X_1(t), \q  \hat{X}^\e_2(t) := \hat{X}^\e_1(t) - X_2(t).
\end{align}
In the following, we  give estimates for $\hat{X}^\e$, $X_1$, $\hat{X}^\e_1$, $X_2$, and $\hat{X}^\e_2$ under the probability measure $\dbQ^*$.
\begin{proposition}\label{state}
Let \autoref{A1} hold. Then, we have 
\begin{align*}
&(i)\q \dbE^{\dbQ^*}\Big[\mathop{\sup}_{t\in [0, T]}|X^*(t)|^4\Big] < \i, \q
(ii)\q\dbE^{\dbQ^*}\Big[\mathop{\sup}_{t\in [0, T]}|\hat{X}^\e(t)|^4\Big] = O(\e^2), \\
&(iii)\q\dbE^{\dbQ^*}\Big[\mathop{\sup}_{t\in [0, T]}|X_1(t)|^4\Big] = O(\e^2),\q
(iv)\q\dbE^{\dbQ^*}\Big[\mathop{\sup}_{t\in [0, T]}|\hat{X}_1^\e(t)|^2\Big] = O(\e^2),\\
&(v)\q\dbE^{\dbQ^*}\Big[\mathop{\sup}_{t\in [0, T]}|X_2(t)|^2\Big] = O(\e^2),\q
(vi)\q\dbE^{\dbQ^*}\Big[\mathop{\sup}_{t\in [0, T]}|\hat{X}_2^\e(t)|\Big] = o(\e).
\end{align*}
\end{proposition}
\begin{proof}
 The proof will be divided into three steps.\\
{\bf Step 1: proof of (i)-(iii).} 
On the one hand, by the definition of $X^*$, we have
\begin{align*}
X^*(t) = x(0) + \int_{0}^{t} \big\{b\big(s, X^*(s), u^*(s)\big) + \sigma\big(s, u^*(s)\big) q^*(s)\big\}ds + \int_{0}^{t}\sigma\big(s, u^*(s)\big)dW^{q^*}(s).
\end{align*}
It then follows from Gronwall's inequality, the Burkholder–Davis–Gundy inequality,  \autoref{A1}, and \autoref{q*} that
\begin{align}\label{q}
\dbE^{\dbQ^*}\Big[\mathop{\sup}_{t\in[0, T]}|X^*(t)|^4 \Big] \leq K \bigg\{\dbE^{\dbQ^*}\Big[\Big(\int_{0}^{T}|q^*(s)|^2ds\Big)^2\Big] + |x(0)|^4  +1\bigg\} < \i.
\end{align}
On the other hand, by the definition of $\hat{X}^\e$, we have
\begin{align*}
\hat{X}^\e(t) =    \int_{0}^{t}\Big\{ \tilde{b}_x^\e(s) \hat{X}^\e(s) + \delta b(s)1_{[t_0, t_0+\e]}(s) + \delta \sigma(s)1_{[t_0, t_0+\e]}(s)q^*(s)\Big\}ds 
+ \int_{0}^{t}\delta \sigma(s)1_{[t_0, t_0+\e]}(s)dW^{q^*}(s),
\end{align*}
where
$\tilde{b}_x^\e(s):=\int_{0}^{1}b_x\big(s, X^*(s)+\theta (X^\e(s) - X^*(s)), u^\e(s)\big)d\theta$.
It then follows from
Gronwall's inequality, the Burkholder–Davis–Gundy inequality,  \autoref{A1},  \eqref{q}, and \autoref{q*} that 
\begin{align*}
\dbE^{\dbQ^*}\Big[\mathop{\sup}_{s\in[0, T]}\frac{|\hat{X}^\e(s)|^4}{\e^2} \Big]
\leq &  4e^{4\|b_x\|_\i  T} \bigg\{\frac{1}{\e^2} \dbE^{\dbQ^*}\Big[\Big(\int_{0}^{T} \Big|  \delta b(s)1_{[t_0, t_0+\e]}(s) +\delta \sigma(s)1_{[t_0, t_0+\e]}(s)q^*(s) \Big|ds\Big)^4\Big] \\
&\q\q\q\q+ \frac{1}{\e^{2}} \dbE^{\dbQ^*}\Big[\mathop{\sup}_{t\in [0, T]}\Big|\int_{0}^{t}\delta \sigma(s)1_{[t_0, t_0+\e]}(s)dW^{q^*}(s)\Big|^4\Big]\bigg\}\\
\leq &  K \bigg\{ \dbE^{\dbQ^*}\Big[
\mathop{\sup}_{s\in[0, T]}|X^*(s)|^4 \Big] + \dbE^{\dbQ^*}\Big[\Big(\int_{0}^{T}|q^*(s)|^2ds\Big)^2\Big] + 1
\bigg\}\\
\leq &K\bigg\{  \dbE^{\dbQ^*}\Big[\Big(\int_{0}^{T}|q^*(s)|^2ds\Big)^2\Big] + |x(0)|^4+ 1
\bigg\} \leq K<\i.
\end{align*}
Consequently, we conclude that (i) and (ii) hold. By a similar deduction, we deduce (iii).\\
{\bf Step 2: proof of (iv) and (v).}
By the definition of $\hat{X}^\e_1$, one has that
\begin{align*}
\hat{X}^\e_1(t) = \int_{0}^{t} \Big\{ \tilde{b}^\e_x(s)\hat{X}_1^\e(s)+ \big(\tilde{b}^\e_x(s) - b_x(s)\big) X_1(s) + \delta b(s)1_{[t_0, t_0+\e]}(s)\Big\}ds.
\end{align*}
It then follows from Gronwall's inequality, \autoref{A1}, and H$\ddot{\rm o}$lder's inequality that 
\begin{align*}
&\dbE^{\dbQ^*}\Big[\mathop{\sup}_{s\in[0, T]}|\hat{X}_1^\e(s)|^2\Big] \\
&\leq Ke^{\|b_x\|_\i T} \bigg\{ \dbE^{\dbQ^*}\Big[ \Big(\int_{0}^{T}\big|\tilde{b}^\e_x(s) - b_x(s)\big|\cd|X_1(s)|ds\Big)^2\Big]  
\q+ \dbE^{\dbQ^*}\Big[ \Big|\int_{0}^{T}|\delta b(s)|1_{[t_0, t_0+\e]}(s)ds\Big|^2\Big]\bigg\}\\
&\leq K\bigg\{\dbE^{\dbQ^*}\Big[ \int_{0}^{T}\big|\hat{X}^\e(s)\big|^2|X_1(s)|^2ds\Big]  +  \dbE^{\dbQ^*}\Big[ \Big(\int_{0}^{T}|\delta b_x(s)|\cd |X_1(s)|1_{[t_0, t_0+\e]}(s)ds\Big)^2\Big]\\
& \q\q\q+ \dbE^{\dbQ^*}\Big[ \Big(\int_{0}^{T}|\delta b(s)|1_{[t_0, t_0+\e]}(s)ds\Big)^2\Big]
\bigg\}
\\
&\leq K \bigg\{\dbE^{\dbQ^*}\Big[\mathop{\sup}_{s\in[0, T]}|\hat{X}^\e(s)|^4\Big]^{\frac{1}{2}}\cd \dbE^{\dbQ^*}\Big[\mathop{\sup}_{s\in[0, T]}|X_1(s)|^4\Big]^{\frac{1}{2}} +  \dbE^{\dbQ^*}\Big[\mathop{\sup}_{s\in[0, T]}|X_1(s)|^2\Big] \e^2\\
&\q\q\q+ \dbE^{\dbQ^*}\Big[\mathop{\sup}_{s\in[0, T]}|X^*(s)|^2 +1\Big] \e^2 \bigg\} \leq K\e^2,
\end{align*}
which implies (iv) holds. Similar to the deduction of (iv), we deduce (v).\\
{\bf Step 3: proof of (vi).} 
By the definition of $\hat{X}_2^\e$, one has that
\begin{align*}
&\hat{X}^\e_2(t) = \int_{0}^{t} \bigg\{ \delta b_x(s)1_{[t_0, t_0+\e]}(s)\hat{X}^\e(s) + b_x(s)\hat{X}^\e_2(s) 
+ \frac{1}{2}\big(\tilde{b}_{xx}^\e(s) - b_{xx}(s)\big) |\hat{X}^\e(s)|^2 \\
&\q\q\q\q\q\q+ \frac{1}{2} b_{xx}(s)\Big(|\hat{X}^\e(s)|^2 - |X_1(s)|^2\Big)
\bigg\}ds,
\end{align*}
where
$\tilde{b}_{xx}^\e(s):= 2\int_{0}^{1}(1-\theta) b_{xx}\big(s, X^*(s)+\theta \hat{X}^\e(s), u^\e(s)\big)d\theta$.
It then follows from Gronwall's inequality and \autoref{A1} that
\begin{align*}
\dbE^{\dbQ^*}\Big[\mathop{\sup}_{s\in[0, T]}|\hat{X}^\e_2(s)|\Big] &\leq K \dbE^{\dbQ^*}\Big[\int_{0}^{T} \Big\{ |\delta b_x(s)|1_{[t_0, t_0+\e]}(s)|\hat{X}^\e(s)| +  \frac{1}{2}\big|\tilde{b}_{xx}^\e(s) - b_{xx}(s)\big| |\hat{X}^\e(s)|^2 \\
&\q\q\q\q\q\q+ \frac{1}{2}|\delta b_{xx}(s)|1_{[t_0, t_0+\e]}(s) |\hat{X}^\e(s)|^2 + \frac{1}{2} |b_{xx}(s)|\Big||\hat{X}^\e(s)|^2 - |X_1(s)|^2\Big|
\Big\}ds\Big].
\end{align*}
On the one hand, by (i)-(v), \autoref{A1}, and H$\ddot{\rm o}$lder's inequality, we can check that
\begin{equation}\label{11}
\begin{aligned}
	&\dbE^{\dbQ^*}\Big[\int_{0}^{T} \Big\{ |\delta b_x(s)|1_{[t_0, t_0+\e]}(s)|\hat{X}^\e(s)|
	+ \frac{1}{2}|\delta b_{xx}(s)|1_{[t_0, t_0+\e]}(s) |\hat{X}^\e(s)|^2  \\
	& \q\q\q\q\q+ \frac{1}{2} |b_{xx}(s)|\Big| |\hat{X}^\e(s)|^2 - |X_1(s)|^2\Big|
	\Big\}ds\Big] = o(\e).
	\end{aligned}\end{equation}
	On the other hand, note that
	\begin{equation}
\begin{aligned}
	&\dbE^{\dbQ^*}\Big[\int_{0}^{T}\big|\tilde{b}_{xx}^\e(s) - b_{xx}(s)\big| |\hat{X}^\e(s)|^2ds\Big] \\
	&\leq K \bigg\{ \dbE^{\dbQ^*}\Big[\int_{0}^{T}\big|\tilde{b}_{xx}^\e(s) - b_{xx}(s, X^*(s), u^\e(s))\big| |\hat{X}^\e(s)|^2ds\Big] + \dbE^{\dbQ^*}\Big[\int_{t_0}^{t_0+\e}\big|\delta b_{xx}(s) \big| |\hat{X}^\e(s)|^2ds\Big]\bigg\}\\
	&\leq K \bigg\{ \dbE^{\dbQ^*}\Big[\int_{0}^{T}\big|\tilde{b}_{xx}^\e(s) - b_{xx}(s, X^*(s), u^\e(s))\big| |\hat{X}^\e(s)|^2ds\Big] + \e^{2}\bigg\}.
	\end{aligned}\end{equation}
	Thanks to \autoref{A1}, H$\ddot{\rm o}$lder's inequality, and the dominated convergence theorem, one can deduce that
	\begin{align}
\dbE^{\dbQ^*}\Big[\int_{0}^{T}\big|\tilde{b}_{xx}^\e(s) - b_{xx}(s, X^*(s), u^\e(s))\big| |\hat{X}^\e(s)|^2ds\Big] = o(\e).
\end{align}
Thus, we obtain that
\begin{align}\label{22}
\dbE^{\dbQ^*}\Big[\int_{0}^{T}\big|\tilde{b}_{xx}^\e(s) - b_{xx}(s)\big| |\hat{X}^\e(s)|^2ds\Big] = o(\e).
\end{align}
Combining \eqref{11} and \eqref{22}, we obtain (vi).

This completes the proof.
\end{proof}
\section{The estimate for $(Y^\e - Y^*, Z^\e - Z^*)$ under the probability measure $\dbQ^*$}\label{section4}
In this section, we present the estimate for $Y^\e - Y^*$ under the probability measure $\dbQ^*$, which is the main technical contribution and difficulty of this paper.

Set 
\begin{align}\label{haty}
\big(\hat{Y}^\e(t), \hat{Z}^\e(t)\big) := \big(Y^\e(t) - Y^*(t), Z^\e(t) - Z^*(t)\big). 
\end{align}
We have the following result.
\begin{proposition}\label{importan-}
Let \autoref{A1} hold. Then, we have that 
\begin{align}
\dbE^{\dbQ^*}\Big[\mathop{\sup}_{t\in [0, T]} |\hat{Y}^\e(t)|^4 \Big] = O(\e^2) \q\hb{and} \q \dbE^{\dbQ^*}\Big[\Big(\int_{0}^{T}|\hat{Z}^\e(t)|^2dt\Big)^2\Big] = O(\e^{2}).
\end{align}
\end{proposition}
\begin{proof}
{\bf Step 1:} Let us first prove that $\dbE^{\dbQ^*}\Big[\mathop{\sup}_{t\in [0, T]} |\hat{Y}^\e(t)|^4 \Big] = O(\e^2)$. Define 
\begin{align*}
& q^\e(s) := f_z\big(s, X^\e(s),  Y^\e(s), Z^\e(s), u^\e(s)\big), \q f^\e(s, y, z) := f\big(s, X^\e(s), y, z, u^\e(s)\big),\\
&\frac{{ d}{\dbQ^\e}}{{ d}\dbP} := \mathscr{E}\Big(\int_{0}^{\cd}f_z\big(s, X^\e(s), Y^\e(s), Z^\e(s), u^\e(s) \big) dW(s)\Big)_T = \mathscr{E}\Big(\int_{0}^{\cd}f^\e_z\big(s, Y^\e(s), Z^\e(s) \big) dW(s)\Big)_T,\\
&dW^{q^\e}(t):= dW(t) - q^\e(t)dt, \q g^\e(t, y, q) := \mathop{\sup}_{z\in \dbR^{d}}\Big(z q - f^\e(t, y, z)\Big).
\end{align*}
By an argument similar to that in \autoref{prop3.2}, we can check that $\dbQ^\e$ is a probability measure and $W^{q^\e}$ is a Brownian motion under $\dbQ^\e$.
For $(q, \dbQ) \in \{(q^*, \dbQ^*), (q^\e, \dbQ^\e)\}$, let $Y^{*q}$, $Y^{\e q}$ be the solutions for the following BSDEs:
\begin{equation}\label{bsd}
\begin{aligned}
	Y^{*q}(r) = h\big(X^{*}(T)\big) + \int_{r}^{T} g^*\big(s, Y^{*q}(s), q(s)\big)ds + \int_{r}^{T}Z^{*q}(s) dW^q(s),\\
	Y^{\e q}(r) = h\big(X^{\e}(T)\big) + \int_{r}^{T} g^\e\big(s, Y^{\e q}(s), q(s)\big)ds + \int_{r}^{T}Z^{\e q}(s) dW^q(s),
	\end{aligned}\end{equation}
	where the existence and the $L^2_\dbF(\dbQ)$-integrability follow from Briand, Delyon, Hu, Pardoux, and Stoica \cite[Theorem 4.1]{briand2003lp}.
	By the proof of Delbaen, Hu, and Richou \cite [line 14, page 567]{delbaen2011uniqueness},
	one can obtain that for $t\in [0, T]$ and $q\in \{q^*, q^\e\},$
	\begin{align*}
	Y^*(t) \leq Y^{* q}(t), \q Y^\e(t) \leq Y^{\e q}(t) \q \hb{and} \q Y^*(t) = Y^{* q^*}(t), \q Y^\e(t) = Y^{\e q^\e}(t),
	\end{align*}
	which implies that for $t\in[0, T]$,
	\begin{equation}\label{important}
\begin{aligned}
	&\frac{|\hat{Y}^\e(t)|^2}{\e} = \frac{1}{\e}\Big\{\big|Y^{\e q^\e}(t) -    Y^{* q^*}(t) \big|^2\Big\}
	=\frac{1}{\e}\Big\{\big|\essinf_{q\in\{q^*, q^\e \}} Y^{\e q}(t) - \essinf_{q\in\{q^*, q^\e\}}   Y^{* q}(t) \big|^2\Big\}\\
	&\leq \frac{1}{\e} \esssup_{q\in \{q^*, q^\e\}} |Y^{\e q}(t) - Y^{* q}(t)|^2.
	%
	%
	\end{aligned}\end{equation}
	By \eqref{2.4} in \autoref{A1}, we deduce that
	\begin{equation}\label{g}
\begin{aligned}
	&\big|g^\e\big(s, Y^{\e q}(s), q\big) - g^*\big(s,Y^{* q}(s), q\big)\big| \\
	&\leq \big|\mathop{\sup}_{z\in \dbR^d} \big\{z q - f^\e(s, Y^{\e q}(s), z)\big\} - \mathop{\sup}_{z\in \dbR^d} \big\{z q - f^*(s, Y^{* q}(s), z)\big\}\big|
	\\
	&\leq \mathop{\sup}_{z\in \dbR^d}\big|f\big(s, X^\e(s), Y^{\e q}(s), z, u^\e(s)\big) - f\big(s, X^*(s), Y^{* q}(s), z, u^*(s)\big)\big|\\
	&\leq \|f_x\|_\i |\hat{X}^\e(s)| + \|f_y\|_\i |Y^{\e q}(s) - Y^{* q}(s)| + L_4\big(1+|X^*(s)| \big) 1_{[t_0, t_0+\e]}(s).
	\end{aligned}\end{equation}
	In addition, by the definition of $Y^{*q}$ and $Y^{\e q}$ in \eqref{bsd}, we have that for $t\in[0, T]$ and $r\in [t, T]$,
	\begin{equation}\label{equation}
\begin{aligned}
	Y^{\e q}(r) - Y^{*q}(r) 
	&= h(X^\e(T)) - h(X^*(T))  \\
	&\q + \int_{r}^{T}\{g^\e(s, Y^{\e q}(s), q(s)) - g^*(s, Y^{*q}(s), q(s)) \}ds+ \int_{r}^{T} \{Z^{ \e q}(s) - Z^{*q}(s)\} dW^{q}(s) .
	\end{aligned}\end{equation}
	On the one hand, for $t\in[0, T]$ and $v\in [t, T]$, applying It$\hat{\rm o}$'s formula to $\big\{\big(Y^{\e q}(r) - Y^{*q}(r)\big)^2\big\}_{r\in [t, v]}$ and taking conditional expectation $\dbE_t^\dbQ[\cd]$, one can deduce that 
	\begin{equation}\label{z}
\begin{aligned}
	&\dbE^\dbQ_t\Big[\int_{t}^{v}|Z^{\e q}(s) - Z^{*q}(s)|^2ds\Big]
	\leq |Y^{\e q}(t) - Y^{* q}(t)|^2 +\dbE^\dbQ_t\Big[\int_{t}^{v}|Z^{\e q}(s) - Z^{*q}(s)|^2ds\Big]
	\\
	& \leq \dbE^{\dbQ}_t\Big[|Y^{\e q}(v) - Y^{* q}(v)|^2\Big] + \dbE^\dbQ_t\Big[\int_{t}^{v}2\Big|\big(Y^{\e q}(s) - Y^{* q}(s)\big)\big\{ g^\e\big(s, Y^{\e  q}(s), q(s)\big) - g^*\big(s,Y^{* q}(s), q(s)\big)\big\}\Big|ds\Big].
	\end{aligned}\end{equation}
	On the other hand, taking  the square, the supremum, and then the expectation $\dbE_t^\dbQ[\cd]$ on both sides of \eqref{equation}, and applying the Burkholder–Davis–Gundy inequality, Young's inequality, we deduce that for $t\in [0, T]$ and $v \in [t, T]$, 
	\begin{equation}\label{above}
\begin{aligned}
	&\dbE^\dbQ_t\Big[\mathop{\sup}_{r\in[v, T]}|Y^{\e q}(r) - Y^{*q}(r)|^2\Big]\\
	& \leq   3\|h_x\|_\i^2\dbE^\dbQ_t\big[|\hat{X}^\e(T)|^2\big]+ 3T\dbE_t^\dbQ\Big[\int_{v}^{T}\Big|g^\e\big(s, Y^{\e q}(s), q(s)\big) - g^*\big(s,Y^{*q}(s), q(s)\big)\Big|^2ds\Big] \\
	&\q+ 3c\dbE^\dbQ_t\Big[\int_{v}^{T}|Z^{\e q}(s) - Z^{*q}(s)|^2ds\Big]\\
	& \leq   3\|h_x\|_\i^2\dbE^\dbQ_t\big[|\hat{X}^\e(T)|^2\big] + (3T+3c)\dbE_t^\dbQ\Big[\int_{v}^{T}\Big|g^\e\big(s, Y^{\e q}(s), q(s)\big) - g^*\big(s,Y^{*q}(s), q(s)\big)\Big|^2ds\Big]\\
	& \q +3c\bigg\{\|h_x\|_\i^2\dbE^\dbQ_t\big[|\hat{X}^\e(T)|^2\big] +  \dbE^\dbQ_t\Big[\int_{v}^{T}\big|Y^{\e q}(s) - Y^{*q}(s)\big|^2ds\Big] \bigg\}
	\\
	& \leq   (3+3c)\|h_x\|_\i^2\dbE^\dbQ_t\big[|\hat{X}^\e(T)|^2\big]+ (12c+9T)(\|f_y\|_\i^2+1) \dbE^\dbQ_t\Big[\int_{v}^{T}\mathop{\sup}_{r\in[s, T]}\big|Y^{\e q}(r) - Y^{*q}(r)\big|^2ds\Big]\\
	&\q+ (6T+6c)\dbE_t^\dbQ\Big[\int_{t}^{T}\Big|\|f_x\|_\i |\hat{X}^\e(s)| + L_4\big(1+|X^*(s)| \big) 1_{[t_0, t_0+\e]}(s)\Big|^2ds\Big],
	\end{aligned}\end{equation}
	where $c$ is a constant related to the Burkholder–Davis–Gundy inequality.
	It then follows from Gronwall's inequality that
	\begin{equation}
\begin{aligned}
	\dbE^\dbQ_t\Big[\mathop{\sup}_{r\in[v, T]}|Y^{\e q}(r) - Y^{*q}(r)|^2\Big] &\leq e^{(12c+9T)(\|f_y\|_\i^2+1)T} \bigg\{(3+3c)\|h_x\|_\i^2\dbE^\dbQ_t\big[|\hat{X}^\e(T)|^2\big]\\ 
	&\q+  (6T+6c)\dbE_t^\dbQ\Big[\int_{t}^{T}\Big|\|f_x\|_\i |\hat{X}^\e(s)| + L_4\big(1+|X^*(s)| \big) 1_{[t_0, t_0+\e]}(s)\Big|^2ds\Big]\bigg\}.
\end{aligned}
\end{equation}
Then, letting $v=t$ in the above inequality, we obtain that for $t\in[0, T]$,
\begin{equation}\label{y}
\begin{aligned}
	\frac{|Y^{\e q}(t) - Y^{* q}(t)|^2}{\e}
	\leq e^{(12c+9T)(\|f_y\|_\i^2+1)T} \bigg\{ &\Big((3+3c)\|h_x\|^2_\i + 3(6T+6c)\|f_x\|_\i^2T\Big)\dbE_t^\dbQ\Big[\mathop{\sup}_{s\in [t, T]} \frac{|\hat{X}^\e(s)|^2}{\e}\Big] \\
	&+ 3(6T+6c)L_4^2\dbE_t^\dbQ\Big[\mathop{\sup}_{s\in [t, T]} |X^*(s)|^2\Big] +3(6T+6c)L_4^2\bigg\}.
	\end{aligned}\end{equation}
	By the definition of $\hat{X}^\e$ and $X^*$, as well as by Gronwall's inequality,  \autoref{A1}, and H$\ddot{\rm o}$lder's inequality, we obtain that
	\begin{equation}\label{xe}
\begin{aligned}
	&\dbE_t^\dbQ\Big[\mathop{\sup}_{s\in [t, T]} \frac{|\hat{X}^\e(s)|^2}{\e}\Big] \\
	&\leq e^{2\|b_x\|_\i T}\bigg\{\frac{3|\hat{X}^\e(t)|^2}{\e} +  3\dbE_t^\dbQ\Big[\Big\{\int_{t}^{T}\Big(\frac{\delta b(s)1_{[t_0, t_0+\e]}(s)}{\sqrt{\e}}+ \frac{\delta \sigma(s)1_{[t_0, t_0+\e]}(s)|q(s)|}{\sqrt{\e}}\Big)ds\Big\}^2\Big] + 36 L^2_1\bigg\} \\
	&\leq e^{2\|b_x\|_\i T}\bigg\{\frac{3|\hat{X}^\e(t)|^2}{\e}+ 48L^2_1(T^2+1) +36\|b_x\|_\i^2T^2\dbE_t^\dbQ\Big[\mathop{\sup}_{s\in [t, T]}|X^*(s)|^2\Big] + 36L^2_1\dbE_t^\dbQ\Big[\int_{t}^{T}|q(s)|^2ds\Big]\bigg\}
	\end{aligned}\end{equation}
	and 
	\begin{align}\label{x}
\dbE_t^\dbQ\Big[\mathop{\sup}_{s\in [t, T]} |X^*(s)|^2\Big] \leq e^{2\|b_x\|_\i T}\bigg\{3|X^*(t)|^2 + 8L_1^2(T^2+1) + 3L_1^2T\dbE_t^\dbQ\Big[\int_{t}^{T} |q(s)|^2ds\Big] \bigg\}.
\end{align}
Then, combining \eqref{important} and \eqref{y}-\eqref{x}, we obtain that
\begin{equation}\label{1}
\begin{aligned}
	&\frac{|\hat{Y}^\e(t)|^2}{\e} 
	\leq  \esssup_{q\in\{q^*, q^\e\}}\bigg\{\frac{A_1|\hat{X}^\e(t)|^2}{\e}
	+ A_2|X^*(t)|^2
	+A_3\dbE_t^\dbQ\Big[\int_{t}^{T}|q(s)|^2ds\Big] + A_4\bigg\}\\
	&\leq \frac{A_1|\hat{X}^\e(t)|^2}{\e}
	+ A_2|X^*(t)|^2
	+A_3\dbE_t^{\dbQ^*}\Big[\int_{t}^{T}|q^*(s)|^2ds\Big] +A_3\dbE_t^{\dbQ^\e}\Big[\int_{t}^{T}|q^\e(s)|^2ds\Big]+ A_4 ,
	\end{aligned}\end{equation}
	where 
	\begin{align*}
&\Lambda_1 := (12c+9T)(\|f_y\|_\i^2+1)T,\q \Lambda_2:=(3+3c)\|h_x\|^2_\i + 3(6T+6c)\|f_x\|_\i^2T,\\
&A_1 := 3e^{ \Lambda_1+2\|b_x\|_\i T}\Lambda_2, \q
A_2:= e^{ \Lambda_1+2\|b_x\|_\i T}\bigg\{9(6T+6c)L_4^2 + 108e^{ 2\|b_x\|_\i T}\Lambda_2\|b_x\|_\i^2T^2\bigg\}, \\
&A_3 :=  e^{\Lambda_1+ 2\|b_x\|_\i T}\bigg\{\Big(36L_1^2+108L_1^2\|b_x\|_\i^2T^3e^{2\|b_x\|_\i T}\Big)\Lambda_2 + 9(6T+6c)L_1^2L_4^2T \bigg\},\\
&A_4 := e^{\Lambda_1}\bigg\{3(6T+6c)L_4^2+24(6T+6c)L_4^2L_1^2(T^2+1)e^{2\|b_x\|_\i T} \\
& \q\q\q\q\q+ \Lambda_2\Big(48L_1^2 (T^2+1)e^{2\|b_x\|_\i T}+ 288L_1^2\|b_x\|^2_\i T^2(T^2+1)e^{4\|b_x\|_\i T}\Big)\bigg\}.
\end{align*}
We next remove the dependence of the probability measure $\dbQ^\e$ in the right side of \eqref{1}.
For this purpose, we will estimate $\dbE_t^{\dbQ^*}\big[\int_{t}^{T}|q^*(s)|^2ds\big]$ and  $\dbE_t^{\dbQ^\e}\big[\int_{t}^{T}|q^\e(s)|^2ds\big].$
According to the proof of Delbaen, Hu, and Richou \cite [line 14, page 567]{delbaen2011uniqueness} and \eqref{ge}, we deduce that
\begin{equation}\label{fin}
\begin{aligned}
	&Y^{*}(t) = \dbE^{\dbQ^*}_t \Big[h(X^*(T)) + \int_{t}^{T}g^*\big(s, Y^{*}(s),  q^*(s)\big)ds\Big]\\
	&\geq -\|h(0)\|_\i-\|h_x\|_\i\dbE_t^{\dbQ^*}\Big[ |X^*(T)|\Big] + \frac{1}{8L_3} \dbE_t^{\dbQ^*}\Big[\int_{t}^{T}|q^*(s)|^2ds\Big] -\|f_y\|_\i \dbE_t^{\dbQ^*}\Big[\int_{t}^{T}|Y^{*}(s)|ds\Big] \\
	&\q- \dbE_t^{\dbQ^*}\Big[\int_{t}^{T}\Big\{L_2+L_3 + \|f_x\|_\i |X^*(r)|\Big\}dr\Big]\\
	&\geq \frac{1}{8L_3} \dbE_t^{\dbQ^*}\Big[\int_{t}^{T}|q^*(s)|^2ds\Big] -\|f_y\|_\i \dbE_t^{\dbQ^*}\Big[\int_{t}^{T}|Y^{*}(s)|ds\Big]\\
	&\q-\Big(\|h_x\|_\i + \|f_x\|_\i T \Big)\dbE_t^{\dbQ^*}\Big[ \mathop{\sup}_{s\in [t, T]}|X^*(s)|\Big] - (L_2 + L_3)T-\|h(0)\|_\i.
	\end{aligned}\end{equation}
	On the one hand, recall that
	\begin{align*}
f(t, x, y, z, u) \leq L_2 + L_3 + \|f_x\|_\i |x| + \|f_y\|_\i |y| + 2L_3|z|^2.
\end{align*}
Then, by Briand and Hu \cite[Proposition 1]{Briand and Hu}, one can obtain that
\begin{align*}
|Y^*(s)| \leq A_5 \Big\{\log \dbE_s\Big[\exp\Big\{A_6 \mathop{\sup}_{r\in [s, T]}|X^*(r) |\Big\} \Big] + A_6L_2T+A_6L_3T + A_6 \|h(0)\|_\i\Big\},
\end{align*}
where $A_5:=\frac{1}{4L_3}, A_6:= 2(\frac{1}{A_5}+1)e^{\|f_y\|_\i T}(\|f_x\|_\i T+ \|h_x\|_\i  + 1)>1$,
which implies that
\begin{equation}
\begin{aligned}\label{st1}
	\dbE_t^{\dbQ^*}\Big[\int_{t}^{T}|Y^*(s)|ds\Big] \leq  A_5\bigg\{&\dbE^{\dbQ^*}_t\Big[\int_{t}^{T}\Big(\log \dbE_s\Big[\exp\Big\{A_6 \mathop{\sup}_{r\in [s, T]}|X^*(r) - X^*(s)|\Big\} \Big]+ A_6|X^*(s)|\Big)ds\Big]  \\
	& + A_6L_2T^2+A_6L_3T^2 + A_6 \|h(0)\|_\i T\bigg\}.
	\end{aligned}\end{equation}
	Note the fact that
	\begin{equation}\label{n}
	\begin{aligned}
	&X^*(r) - X^*(s) = \int_{s}^{r}\{b(k, X^*(k), u^*(k)) - b(k, 0, u^*(k))\}dk + \int_{s}^{r}b(k, 0, u^*(k))dk + \int_{s}^{r}\sigma(k, u^*(k))dW(k)\\
	&\leq \int_{s}^{r}\|b_x\|_\i |X^*(k) - X^*(s)|dk + (\|b_x\|_\i |X^*(s)| + L_1)T + \int_{s}^{r}\sigma(k, u^*(k))dW(k), \q t\leq s\leq r \leq T.
	\end{aligned}\end{equation}
	Using \eqref{n}, Gronwall's inequality, Doob's maximal inequality, and H$\ddot{\rm o}$lder's inequality, we deduce that
	\begin{align}
&\dbE_s\Big[\exp\Big\{A_6\mathop{\sup}_{r\in [s, T]}|X^*(r) - X^*(s) |\Big\}\Big]\nonumber \\
&\leq  \dbE_s\Big[\mathop{\sup}_{r\in[s, T]}\exp\Big\{A_6 e^{\|b_x\|_\i T}\Big|\int_{s}^{r}\sigma(k, u^*(k))dW(k)\Big| \Big\} \Big]
\cd\exp\Big\{A_6e^{\|b_x\|_\i T} \big(\|b_x\|_\i|X^*(s)|+L_1\big)T\Big\}\nonumber\\
&\leq \bigg\{\dbE_s\Big[\mathop{\sup}_{r\in[s, T]}\exp\Big\{A_6 e^{\|b_x\|_\i T}\int_{s}^{r}\sigma(k, u^*(k))dW(k) \Big\}  + \mathop{\sup}_{r\in[s, T]}\exp\Big\{-A_6 e^{\|b_x\|_\i T}\nonumber\\
&  \q\q\cd \int_{s}^{r}\sigma(k, u^*(k))dW(k) \Big\}\Big]\bigg\} \exp\Big\{A_6e^{\|b_x\|_\i T} \big(\|b_x\|_\i|X^*(s)|+L_1\big)T\Big\}\nonumber
\\
&\leq \bigg\{\dbE_s\Big[\mathop{\sup}_{r\in[s, T]}\mathscr{E}\Big(\int_{s}^{\cd}\sigma(k, u^*(k))dW(k) \Big)_r^{A_6 e^{\|b_x\|_\i T}}  + \mathop{\sup}_{r\in[s, T]}\mathscr{E}\Big(-\int_{s}^{\cd}\sigma(k, u^*(k))dW(k) \Big)_r^{A_6 e^{\|b_x\|_\i T}}\Big]\bigg\}\nonumber\\
&  \q\cd \exp\Big\{L^2_1TA_6 e^{\|b_x\|_\i T} + A_6e^{\|b_x\|_\i T} \big(\|b_x\|_\i|X^*(s)|+L_1\big)T\Big\}\nonumber
\\
&\leq \bigg\{\Big(\frac{d'}{d'-1}\Big)^{d'}\dbE_s\Big[\mathscr{E}\Big(\int_{s}^{\cd}\sigma(k, u^*(k))dW(k) \Big)_T^{A_6 e^{\|b_x\|_\i T}}  + \mathscr{E}\Big(-\int_{s}^{\cd}\sigma(k, u^*(k))dW(k) \Big)_T^{A_6 e^{\|b_x\|_\i T}}\Big]\bigg\}\nonumber\\
&  \q\cd \exp\Big\{L^2_1TA_6 e^{\|b_x\|_\i T} + A_6e^{\|b_x\|_\i T} \big(\|b_x\|_\i|X^*(s)|+L_1\big)T\Big\}
\nonumber\\
&\leq \bigg\{\Big(\frac{d'}{d'-1}\Big)^{d'}\dbE_s\Big[\mathscr{E}\Big(\int_{s}^{\cd}A_6 e^{\|b_x\|_\i T}\sigma(k, u^*(k))dW(k) \Big)_T  + \mathscr{E}\Big(-\int_{s}^{\cd}A_6 e^{\|b_x\|_\i T}\sigma(k, u^*(k))dW(k) \Big)_T\Big]\bigg\}\nonumber\\
&  \q\cd \exp\Big\{\big( |A_6e^{\|b_x\|_\i T}|^2 - A_6e^{\|b_x\|_\i T} \big)L_1^2T
+L^2_1TA_6 e^{\|b_x\|_\i T} + A_6e^{\|b_x\|_\i T} \big(\|b_x\|_\i|X^*(s)|+L_1\big)T\Big\}, \label{m}
\end{align}
where $d':= A_6 e^{\|b_x\|_\i T}.$
Since $\Big\{\mathscr{E}\Big(\int_{s}^{\cd}A_6 e^{\|b_x\|_\i T}\sigma(k, u^*(k))dW(k) \Big)_r\Big\}_{r\in [s, T]}$ is a martingale, we deduce that
\begin{align}\label{mm}
\dbE_s\Big[\mathscr{E}\Big(\pm\int_{s}^{\cd}A_6 e^{\|b_x\|_\i T}\sigma(k, u^*(k))dW(k) \Big)_T\Big] = \mathscr{E}\Big(\pm\int_{s}^{\cd}A_6 e^{\|b_x\|_\i T}\sigma(k, u^*(k))dW(k) \Big)_s = 1.
\end{align}
Then, combining \eqref{m} and \eqref{mm}, \eqref{st1} can be changed to 
\begin{equation}
\begin{aligned}\label{st2}
	\dbE_t^{\dbQ^*}\Big[\int_{t}^{T}|Y^*(s)|ds\Big] \leq  A_5&\dbE^{\dbQ^*}_t\Big[A_7\mathop{\sup}_{r\in [t, T]}|X^*(r)| + A_8\Big],
	\end{aligned}\end{equation}
	where 
	\begin{align*}
&A_7:= A_6e^{\|b_x\|_\i T}\|b_x\|_\i T^2 + A_6T, \\
&A_8 :=T\log\Big[2\Big(\frac{d'}{d'-1}\Big)^{d'}\Big]+(|A_6e^{\|b_x\|_\i T}|^2 L_1^2T^2
+ A_6e^{\|b_x\|_\i T} L_1T^2)+A_6T(L_2T+L_3T+\|h(0)\|_\i).
\end{align*}
On the other hand, let $\nu$ be a constant, which will be chosen below. Then, by the definition of $X^*$ and H$\ddot{\rm o}$lder's inequality, we deduce that
\begin{equation}\label{x*}
\begin{aligned}
	&\dbE_t^{\dbQ^*}\Big[\mathop{\sup}_{r\in [t, T]} |X^*(r)|\Big] \leq e^{\|b_x\|_\i T}\bigg\{|X^*(t)|  + L_1\dbE_t^{\dbQ^*}\Big[\int_{t}^{T} |q^*(s)|ds\Big] + 2L_1(T+1)\bigg\} \\
	&\leq e^{\|b_x\|_\i T}\bigg\{|X^*(t)|  + L_1\nu^2\dbE_t^{\dbQ^*}\Big[\int_{t}^{T} |q^*(s)|^2ds\Big] + \frac{TL_1}{\nu^2} + 3L_1(T+1) \bigg\}.
	\end{aligned}\end{equation}
	It then follows from \eqref{st2}, \eqref{x*}, and \eqref{fin} that 
	\begin{align*}
A_9\dbE_t^{\dbQ^*}\Big[\int_{t}^{T}|q^*(s)|^2ds\Big] \leq |Y^*(t)| + A_{10} |X^*(t)| + A_{11},
\end{align*}
where
\begin{align*}
&\Lambda_3:=\|f_y\|_\i A_5 A_7 + \|h_x\|_\i + \|f_x\|_\i  T,\q
A_9:= \frac{1}{8L_3}-\Lambda_3L_1 e^{\|b_x\|_\i T}\nu^2, \q
A_{10} := \Lambda_3 e^{\|b_x\|_\i T},\\
&A_{11}:= \Big(3L_1(T+1) + \frac{TL_1}{\nu^2}\Big)\Lambda_3 e^{\|b_x\|_\i T} + (L_2+L_3)T+\|h(0)\|_\i+\|f_y\|_\i A_5A_8.
\end{align*}
Choose $\nu$ so that $A_9 > 0$. Then we obtain 
\begin{align}\label{2}
\dbE_t^{\dbQ^*}\Big[\int_{t}^{T}|q^*(s)|^2ds\Big] \leq \frac{1}{A_9}\Big\{|Y^*(t)| + A_{10} |X^*(t)| + A_{11}\Big\}.
\end{align}
Similarly, we obtain that
\begin{align}\label{3}
\dbE_t^{\dbQ^\e}\Big[\int_{t}^{T}|q^\e(s)|^2ds\Big] \leq \frac{1}{A_9}\Big\{|Y^\e(t)| + A_{10} |X^\e(t)| + A_{11}\Big\}.
\end{align}
Combining  \eqref{1}, \eqref{2}, and \eqref{3}, we obtain
\begin{equation}\label{key}
\begin{aligned}
	&\frac{|\hat{Y}^\e(t)|^2}{\e} \\
	&\leq \frac{A_1|\hat{X}^\e(t)|^2}{\e}
	+ A_2|X^*(t)|^2
	+\frac{A_3}{A_9}\Big\{ |Y^{*}(t)| + |Y^{\e}(t)| + A_{10}|X^*(t)| + A_{10}|X^\e(t)| +2A_{11} \Big\}+ A_4.
	\end{aligned}\end{equation}
	It then follows from \autoref{state} and \autoref{ystate} that
	\begin{align*}
\dbE^{\dbQ^*}\Big[\mathop{\sup}_{t\in[0, T]}\frac{|\hat{Y}^\e(t)|^4}{\e^2}\Big] \leq& K\bigg\{1+\dbE^{\dbQ^*}\Big[\mathop{\sup}_{t\in[0, T]}\frac{|\hat{X}^\e(t)|^4}{\e^2}\Big] + \dbE^{\dbQ^*}\Big[\mathop{\sup}_{t\in[0, T]}|X^*(t)|^4\Big] \\
&\q\q+  \dbE^{\dbQ^*}\Big[\mathop{\sup}_{t\in[0, T]}\Big\{|X^*(t)|^2 + |X^\e(t)|^2 + |Y^*(t)|^2 + |Y^\e(t)|^2\Big\}\Big]\bigg\} \leq K < \i.
\end{align*}
%
%
Therefore, we have proved that $\dbE^{\dbQ^*}\Big[\mathop{\sup}_{t\in [0, T]}|\hat{Y}^\e(t)|^4\Big] = O(\e^2)$. \\
{\bf Step 2:} In this step, we will estimate $\dbE^{\dbQ^*}\Big[\Big(\int_{0}^{T}|\hat{Z}^\e(t)|^2dt\Big)^2\Big]$.
Let $\tau_n := \inf\big\{t: \int_{0}^{t}|\hat{Z}^\e(s)|^2ds \geq n\big\} \wedge T$.
Then, by the definition of $Y^\e$ and $Y^*$, we obtain that 
\begin{equation}\label{es}
\begin{aligned}
	\hat{Y}^\e(0) 
	= \hat{Y}^\e(\tau_n) - \int_{0}^{\tau_n}\big\{e(s) + r(s)\big\} ds 
	+ \int_{0}^{\tau_n} \hat{Z}^\e(s) dW^{q^*}(s),
	\end{aligned}\end{equation}
	where
	\begin{align}
&e(s):= f\big(s, X^\e(s), Y^\e(s), Z^\e(s), u^\e(s)\big) - f\big(s, X^*(s), Y^*(s), Z^\e(s), u^*(s)\big), \nonumber\\
&r(s):= f\big(s, X^*(s), Y^*(s), Z^\e(s), u^*(s)\big) - f\big(s, X^*(s), Y^*(s), Z^*(s), u^*(s)\big)- f_z(s)\hat{Z}^\e(s).\label{rr}
\end{align}
Applying Taylor's expansion to \eqref{rr} and using \autoref{A1}, we deduce that
\begin{align}
&|e(s)| \leq \|f_x\|_\i |\hat{X}^\e(s)| + \|f_y\|_\i |\hat{Y}^\e(s)| + L_4 (1+|X^*(s)|)1_{[t_0, t_0+\e]}(s),\label{e}\\
& 0 \leq r(s)\leq \frac{\|f_{zz}\|_\i}{2} |\hat{Z}^\e(s)|^2. \label{r}
\end{align}
Define
$
R(t):= \int_{0}^{t}r(s)ds.
$
Noting the relationship that
\begin{align}
R(t) - \int_{0}^{t}\hat{Z}^\e(s)dW^{q^*}(s) = 
\hat{Y}^\e(t) - \hat{Y}^\e(0) - \int_{0}^{t}e(s)ds.
\end{align}
and using \autoref{A1}, \autoref{state},  Step 1, one can deduce that
\begin{align}\label{ii}
\dbE^{\dbQ^*}\Big[\mathop{\sup}_{s\in[0, T]}\Big|R(s) - \int_{0}^{s}\hat{Z}^\e(k)dW^{q^*}(k)\Big|^4\Big] \leq K\e^2.
\end{align}
In addition, applying It$\hat{\rm o}$'s formula to $\Big\{R(t)^3\big\{R(t) - \int_{0}^{t}\hat{Z}^\e(s)dW^{q^*}(s)\big\}\Big\}_{t\in [0, \tau_n]}$ and noting that $R^3$ is a continuous finite-variation process, one can deduce that
\begin{align}\label{r4}
\frac{1}{4}\dbE^{\dbQ^*}\big[|R(\tau_n)|^4\big] = \dbE^{\dbQ^*}\Big[R(\tau_n)^3 \Big\{R(\tau_n) - \int_{0}^{\tau_n}\hat{Z}^\e(s)dW^{q^*}(s)\Big\} - \int_{0}^{\tau_n}\Big\{R(s) - \int_{0}^{s}\hat{Z}^\e(k)dW^{q^*}(k)\Big\}dR(s)^3\Big],
\end{align}
where the definition of $\tau_n$ and the right side of \eqref{r} have been used to handle the stochastic integral term.
Then, using \eqref{ii} and H$\ddot{\rm o}$lder's inequality, we obtain that
\begin{align*}
\dbE^{\dbQ^*}\big[|R(\tau_n)|^4\big] \leq K\dbE^{\dbQ^*}\Big[\mathop{\sup}_{s\in[0, T]}\Big|R(s) - \int_{0}^{s}\hat{Z}^\e(k)dW^{q^*}(k)\Big|^4\Big] \leq K\e^2.
\end{align*}
Letting $n \rightarrow \i$ and using Fatou's lemma, one can obtain $\dbE^{\dbQ^*}\big[|R(T)|^4\big] \leq K\e^2$.
Finally, using the Burkholder–Davis–Gundy inequality, we deduce that
\begin{align*}
&\dbE^{\dbQ^*}\Big[\Big(\int_{0}^{T}|\hat{Z}^\e(t)|^2dt\Big)^2\Big] \leq K\dbE^{\dbQ^*}\Big[\Big|\int_{0}^{T}\hat{Z}^\e(t)dW^{q^*}(t)\Big|^4\Big] \\
&\leq K\Big\{\dbE^{\dbQ^*}\big[|R(T)|^4\big]  + \dbE^{\dbQ^*}\Big[\mathop{\sup}_{s\in[0, T]}\Big| \int_{0}^{s}\hat{Z}^\e(k)dW^{q^*}(k)-R(s) \Big|^4\Big\} \leq K\e^2.
\end{align*}

This completes the proof.
\end{proof}

\begin{remark}
Since $h$ is unbounded, the existing literature does not guarantee that $\int_{0}^{\cd}f_z(s) dW(s)$ has the BMO property. Thus, we do not have a result analogous to that in Hu, Ji, and Xu \cite[Proposition 3.1]{hu2022global}, which causes considerable difficulties in our analysis. 
\end{remark}

\section{The adjoint equations and the variational equation  } \label{section5}
In this section, we study the first-order and second-order adjoint equations, and the variational equation. 

By applying It$\hat{\rm o}$'s formula to $O_1$ and a second-order Taylor's expansion to $h, V_3$ in the derivation following \eqref{tilde-y2}, and by appropriately choosing the generators of the adjoint equations and the variational equation to eliminate the terms $X_1+X_2$ and $(X_1)^2$, we can determine these equations. 
Since the derivation follows exactly the method developed in Hu \cite{Hu}, we directly present the adjoint equations in \eqref{first ad}, \eqref{second ad} and the variational equation in \eqref{va}.
The following are the adjoint equations:
\begin{align}
&P_1(t) = h_x\big(X^*(T)\big) - \int_{t}^{T} \Big\{ \Big(f_y (s) - b_x(s)\Big)P_1(s) + f_z(s)Q_1(s) + f_x(s)\Big\}ds + \int_{t}^{T}Q_1(s)dW(s),\label{first ad}\\  
&P_2(t) = h_{xx}\big(X^*(T)\big) - \int_{t}^{T}\Big\{ \Big(f_y(s) - 2b_x(s)\Big) P_2(s) + f_z(s)Q_2(s) - b_{xx}(s)P_1(s)  \nonumber\\
&\q\q\q\q\q\q\q\q\q\q\q\q\q+ [1, P_1(s), Q_1(s)]D^2f(s) [1, P_1(s), Q_1(s)]^\top \Big\}ds + \int_{t}^{T}Q_2(s)dW(s),\label{second ad}
\end{align}
where  $D^2f$ is the Hessian matrix of $f$ with respect to $(x, y, z)$.
%
Recall that $\dbQ^* $ and $W^{q^*}$ are defined in \autoref{prop3.2} and denote
\begin{align*}
&BMO(\dbQ^*):=
 \Big\{Z:\Om\times [0,T] \to\dbH\Bigm| Z~ \hb{is $\dbF$-progressively measurable, 
	and~} \\
& \quad \quad \quad \quad \quad \quad\quad\ \quad \quad \quad   \|Z\|_{BMO(\dbQ^*)} := \mathop{\sup}\limits_{\tau \in \mathscr{T}[0,T]} \Big\| \dbE^{\dbQ^*}_{\tau}\Big[\int_\tau^T|Z(s)|^2ds\Big]\Big\|^{\frac{1}{2}}_{\i}   <\i \Big\},
\end{align*}
where $\mathscr{T}[0,T]$ denotes the set of all $\dbF$-stopping times $\tau$ with values in $[0,T]$, and $\dbE^{\dbQ^*}_{\tau}$ is the conditional expectation under $\dbQ^*$ with respect to the $\sigma$-field $\sF_{\tau}$. 

In the following, we will present the well-posedness result of adjoint equations.
\begin{proposition}\label{wellpose of ad}
Let \autoref{A1} hold. Then,  for $\beta > 1$,  \eqref{first ad} admits a unique solution $\big(P_1, Q_1\big) \in S^\i_{\dbF}(\dbQ^*) \times L_{\dbF}^\beta(\dbQ^*) $
and \eqref{second ad} admits a unique solution $\big(P_2, Q_2\big) \in S_{\dbF}^\i(\dbQ^*) \times L_{\dbF}^\beta(\dbQ^*)$. 
\end{proposition}
\begin{proof}
	On the one hand,
note that \eqref{first ad} can be rewritten as
\begin{equation}
\begin{aligned}
	-P_1(t) = -h_x\big(X^*(T)\big) + \int_{t}^{T} \Big\{ \Big(-f_y (s) + b_x(s)\Big)\big(-P_1(s)\big) + f_x(s)\Big\}ds - \int_{t}^{T}Q_1(s)dW^{q^*}(s).
\end{aligned}
\end{equation}
It then follows from Fan, Hu, and Tang \cite[Theorem 2.4]{fan2023multi}  that for $\beta>1$, $P_1 \in S^\i_{\dbF}(\dbQ^*)$ and  $Q_1 \in BMO(\dbQ^*)$. Furthermore, by the energy inequality, we obtain that $Q_1 \in L^\beta_{\dbF}(\dbQ^*)$, for $\beta > 1$.

On the other hand, note that \eqref{second ad} can be rewritten as
\begin{align*}
&-P_2(t) = -h_{xx}\big(X^*(T)\big) + \int_{t}^{T}\Big\{ \Big(-f_y(s) + 2b_x(s)\Big) \big(-P_2(s) \big) - b_{xx}(s)P_1(s)  \nonumber\\
&\q\q\q\q\q\q\q\q\q\q\q\q\q+ [1, P_1(s), Q_1(s)] D^2f(s) [1, P_1(s), Q_1(s)]^\top \Big\}ds - \int_{t}^{T}Q_2(s)dW^{q^*}(s).
\end{align*}
Then, by Briand, Delyon, Hu, Pardoux, and Stoica \cite{briand2003lp}, \autoref{A1}, and the fact that $\big(P_1, Q_1\big) \in S^\i_{\dbF}(\dbQ^*) \times L_{\dbF}^\beta(\dbQ^*), \beta>1$, we obtain that for $\beta > 1$,
\eqref{second ad} admits a unique solution $\big(P_2, Q_2\big) \in S_{\dbF}^\beta(\dbQ^*) \times L_{\dbF}^\beta(\dbQ^*)$.
Furthermore, by applying It$\hat{\rm o}$'s formula, we obtain that
\begin{align*}
-P_2(t) = \dbE_t^{\dbQ^*}\Big[\Pi(t,T)\Big(-h_{xx}\big(X^*(T)\big) \Big) + \int_{t}^{T}\Pi(t,s)\phi(s)ds \Big],
\end{align*}
where 
\begin{align*}
&\Pi(t, s):= \exp\Big\{\int_{t}^{s}\Big(-f_y(r)+2b_x(r)\Big)dr\Big\},\\
&\phi (s):= - b_{xx}(s)P_1(s)+ [1, P_1(s), Q_1(s)] D^2f(s) [1, P_1(s), Q_1(s)]^\top.
\end{align*}
It then follows from \autoref{A1} that
\begin{align*}
\|P_2\|_{S^\i_{\dbF}(\dbQ^*)} \leq K \Big\{1+ \mathop{\sup}\limits_{\tau \in \mathscr{T}[0,T]}\Big\|\dbE_\tau^{\dbQ^*}\Big[\int_{\tau}^{T}|Q_1(s)|^2ds\Big]\Big\|_{\i}\Big\} < \i,
\end{align*}
which implies the desired result.

\end{proof}
Next, we give the variational equation and its estimate. 
\begin{equation}\label{va}
\begin{aligned}
\hat{Y}(t) = & - \int_{t}^{T} \Big\{f_y(s)\hat{Y}(s) + f_z(s)\hat{Z}(s) + I(s)\Big\}ds+\int_{t}^{T}\hat{Z}(s) dW(s),
\end{aligned}\end{equation}
where
\begin{equation}\label{i}
\begin{aligned}
I(s):=& \Big[-P_1(s)\delta b(s) + Q_1(s)\delta \sigma(s) - \frac{1}{2}P_2(s)\big(\delta\sigma(s)\big)^2\\
&\q\q\q\q+ f\big(s, X^*(s), Y^*(s), Z^*(s)-P_1(s)\delta \sigma(s), u(s) \big) - f(s)\Big]1_{[t_0, t_0+\e]}(s).
\end{aligned}
\end{equation}
\begin{remark}\label{indep}
If $\sigma$ depends on the state variable $x$, then $f_z$ will appear in the coefficient of $P_1$. Consequently, we need the result that $\dbE^{\dbQ^*}\big[e^{\int_{0}^{T}f_z(s)ds}\big] < \i$, which cannot be directly guaranteed by the estimates established in this paper.
\end{remark}

\begin{proposition}
Let \autoref{A1} hold. Then, we have that
\begin{align*}
\dbE^{\dbQ^*}\Big[\mathop{\sup}_{t\in [0, T]} |\hat{Y}(t)|^2 + \int_{0}^{T}|\hat{Z}(t)|^2dt \Big] = O(\e^{2}).
\end{align*}
\end{proposition}
\begin{proof}
Note that \eqref{va} can be rewritten as
\begin{equation}
	\begin{aligned}
		-\hat{Y}(t) = &  \int_{t}^{T} \bigg\{-f_y(s)\big(-\hat{Y}(s)\big)  + \Big[-P_1(s)\delta b(s) + Q_1(s)\delta \sigma(s) - \frac{1}{2}P_2(s)\big(\delta\sigma(s)\big)^2\\
		&\q\q+ f\big(s, X^*(s), Y^*(s), Z^*(s)-P_1(s)\delta \sigma(s), u(s) \big) - f(s)\Big]1_{[t_0, t_0+\e]}(s)\bigg\}ds-\int_{t}^{T}\hat{Z}(s) dW^{q^*}(s).
\end{aligned}\end{equation}
By \autoref{A1}, we deduce that
\begin{align*}
|f\big(s, X^*(s), Y^*(s), Z^*(s)-P_1(s)\delta \sigma(s), u(s) \big) - f(s)| \leq K\big\{1+|X^*(s)| +|Z^*(s)|\big\}.
\end{align*}
It then follows from Briand, Delyon, Hu, Pardoux, and Stoica \cite{briand2003lp} as well as \autoref{A1}, \autoref{wellpose of ad}, \autoref{state}, \autoref{system-wellposedness}, \autoref{q*} and the Lebesgue differentiation theorem  that 
\begin{align*}
&\frac{1}{\e^2}\dbE^{\dbQ^*}\Big[\mathop{\sup}_{t\in [0, T]} |\hat{Y}(t)|^2 + \int_{0}^{T}|\hat{Z}(t)|^2dt \Big] 
\leq \frac{K}{\e^2} \dbE^{\dbQ^*}\Big[ \Big(\int_{t_0}^{t_0+\e}\big\{1  +  |Q_1(s)|+ |P_2(s)| + |X^*(s)| + |Z^*(s)|\big\}ds\Big)^2\Big]\\
&\leq K+ \frac{K}{\e} \dbE^{\dbQ^*}\Big[ \int_{t_0}^{t_0+\e}\{|Q_1(s)|^2+ |Z^*(s)|^2\}ds\Big] \leq K_{t_0},
\end{align*}
which implies the desired result.
\end{proof}

Recall that $\big(P_1, Q_1\big)$ and $\big(P_2, Q_2\big)$ are defined in \eqref{first ad} and \eqref{second ad} and set 
\begin{align}
	&\big(M_1(t), N_1(t)\big):= \big(P_1(t)X_1(t),  Q_1(t)X_1(t)-P_1(t)\delta \sigma (t)1_{[t_0, t_0+\e]}(t)\big), \q \tilde{N}_1(t):= Q_1(t) X_1(t);\label{11111}\\
	&\big(M_2(t), N_2(t)\big):= \big(P_2(t)\big(X_1(t)\big)^2,  Q_2(t)\big(X_1(t)\big)^2-2P_2(t)X_1(t)\delta \sigma (t)1_{[t_0, t_0+\e]}(t)\big),  \tilde{N}_2 (t):= Q_2(t) \big(X_1(t)\big)^2;\label{22222}\\
	&\big(\tilde{M}(t), \tilde{N}(t)\big):=\big( P_1(t)X_2(t), Q_1(t)X_2(t)\big). \label{33333}
\end{align}
Since we only have an $S^4_{\dbF}(\dbQ^*)$  estimate  for $X_1$ and an $S^2_{\dbF}(\dbQ^*)$ estimate  for $X_2$  in \autoref{state}, we need the following result to estimate $Q_1 X_1, Q_1 X_2$, and $Q_2 (X_1)^2$. These estimates will be used to prove \autoref{prop3.9} and \autoref{last}. 
\begin{proposition}\label{xx}
	Let \autoref{A1} hold. Then, we have that
	\begin{align}
		&(i)\q \dbE^{\dbQ^*}\Big[\mathop{\sup}_{t\in[0, T]}|M_1(t)|^4 + \Big(\int_{0}^{T}|N_1(t)|^2dt\Big)^2\Big] = O(\e^2), \q
		\dbE^{\dbQ^*}\Big[ \Big(\int_{0}^{T}|\tilde{N}_1(t)|^2dt\Big)^2\Big] = O(\e^2);\label{111}\\
		& (ii)\q\dbE^{\dbQ^*}\Big[\mathop{\sup}_{t\in[0, T]}|M_2(t)|^2 + \int_{0}^{T}|N_2(t)|^2dt\Big] = O(\e^2), \q \dbE^{\dbQ^*}\Big[ \int_{0}^{T}|\tilde{N}_2(t)|^2dt\Big] = O(\e^2); \label{222}
		\\
		& (iii)\q\dbE^{\dbQ^*}\Big[\mathop{\sup}_{t\in[0, T]}|\tilde{M}(t)|^2 + \int_{0}^{T}|\tilde{N}(t)|^2dt\Big] = O(\e^2)\label{333}.
	\end{align}
\end{proposition}
\begin{proof}
	(i) Applying It$\hat{\rm o}$'s formula to $\big(P_1(t)X_1(t)\big)_{t\in[0, T]}$, one can deduce that
	\begin{align*}
		M_1(t) = P_1(T)X_1(T) -\int_{t}^{T}\Big\{&f_y(s)M_1(s) + f_z(s)P_1(s)\delta \sigma(s)1_{[t_0, t_0+\e]}(s) + f_x(s)X_1(s)\\
		&
		-Q_1(s)\delta\sigma(s)1_{[t_0, t_0+\e]}(s)\Big\}ds + \int_{t}^{T}N_1(s)dW^{q^*}(s).
	\end{align*}
	It then follows from Briand, Delyon, Hu, Pardoux, and Stoica \cite{briand2003lp}, \autoref{q*}, H$\ddot{\rm o}$lder's inequality, \autoref{state} (iii) and \autoref{wellpose of ad} that
	\begin{align*}
		&\dbE^{\dbQ^*}\Big[\mathop{\sup}_{s\in[0, T]}|M_1(s)|^4 + \Big(\int_{0}^{T}|N_1(s)|^2ds\Big)^2\Big] \\
		& \leq K\dbE^{\dbQ^*}\Big[\mathop{\sup}_{s\in [0, T]}|X_1(s)|^4+\Big(\int_{0}^{T}\big\{|q^*(s)|+|Q_1(s)|\big\}1_{[t_0, t_0+\e]}(s) ds\Big)^4\Big] \leq K\e^2,
	\end{align*}
	which implies the first part of \eqref{111}. In view of the definitions of $N_1$ and $\tilde{N}_1$, one can deduce the second part of \eqref{111}. 
	
	(ii) Applying It$\hat{\rm o}$'s formula to $\big(P_2(t)(X_1(t))^2\big)_{t\in[0, T]}$, one can deduce that
	\begin{align*}
		&M_2(t) = P_2(T)\big(X_1(T)\big)^2 - \int_{t}^{T} \bigg\{f_y(s)M_2(s) + 2 f_z(s)P_2(s)X_1(s)\delta\sigma(s)1_{[t_0, t_0+\e]}(s) + P_2(s)\big(\delta\sigma(s)\big)^21_{[t_0, t_0+\e]}(s) \\
		&+ \big[f_{xx}(s)+f_{yy}(s)\big(P_1(s)\big)^2+ 2f_{xy}(s)P_1(s)- b_{xx}(s)P_1(s) \big]\big(X_1(s)\big)^2 + f_{zz}(s)\big(Q_1(s)X_1(s)\big)^2 + 2f_{xz}(s)Q_1(s)\\
		&
		\cd\big(X_1(s)\big)^2+ 2f_{yz}(s)P_1(s)Q_1(s)\big(X_1(s)\big)^2 - 2Q_2(s)X_1(s)\delta\sigma(s)1_{[t_0, t_0+\e]}(s) \bigg\}ds + \int_{t}^{T}N_2(s)dW^{q^*}(s).
	\end{align*}
	%
	%
		Since 
		\begin{align*}
			&\dbE^{\dbQ^*}\Big[\Big(\int_{t_0}^{t_0+\e}|f_z(s)P_2(s)X_1(s)\delta\sigma(s)|ds\Big)^2\Big] \leq \dbE^{\dbQ^*}\Big[\mathop{\sup}_{s\in [0, T]}|X_1(s)|^2\Big(\int_{t_0}^{t_0+\e}|q^*(s)|ds\Big)^2\Big] \\
			& \leq \dbE^{\dbQ^*}\Big[\mathop{\sup}_{s\in [0, T]}|X_1(s)|^2\Big(\int_{t_0}^{t_0+\e}|q^*(s)|^2ds\Big)\Big]\e\leq \dbE^{\dbQ^*}\Big[\mathop{\sup}_{s\in [0, T]}|X_1(s)|^4\Big]^\frac{1}{2}\dbE^{\dbQ^*}\Big[\Big(\int_{t_0}^{t_0+\e}|q^*(s)|^2ds\Big)^2\Big]^\frac{1}{2}\e \\
			&\leq K\e^2,
		\end{align*}
		one can deduce  that the first part of \eqref{222} holds.
	Furthermore, in view of the definitions of $N_2$ and $\tilde{N}_2$, one can deduce the second part of \eqref{222}.

	(iii) Applying It$\hat{\rm o}$'s formula to $\big(P_1(t)X_2(t)\big)_{t\in [0, T]}$, one can deduce that
	\begin{align*}
		\tilde{M}(t) = P_1(T)X_2(T) - \int_{t}^{T}\Big\{&f_y(s) \tilde{M}(s)+ P_1(s)\delta b(s)1_{[t_0, t_0+\e]}(s)+ \frac{1}{2}P_1(s)b_{xx}(s)\big(X_1(s)\big)^2\\
		&+f_x(s)X_2(s)\Big\}ds + \int_{t}^{T}\tilde{N}(s)dW^{q^*}(s).
	\end{align*}
	It then follows from \autoref{A1}, \autoref{state}, and \autoref{wellpose of ad} that \eqref{333} holds.
	
	This completes the proof.
\end{proof}
\section{The asymptotic expansion for $Y^\varepsilon-Y^*$}\label{section8}
In this section, we derive an asymptotic expansion of $Y^\varepsilon-Y^*$ and estimate the terms in the expansion.
\subsection{The estimate for $(\hat{Y}_1^\e, \hat{Z}_1^\e)$}
In this subsection, we give an $L^2_{\dbF}(\dbQ^*)$ estimate for $\big(\hat{Y}_1^\e, \hat{Z}_1^\e\big)$ in \eqref{y1}, which will be used to prove \autoref{last}.

Recall that $\big(\hat{Y}^\e(t), \hat{Z}^\e(t)\big) := \big(Y^\e(t) - Y^*(t), Z^\e(t) - Z^*(t) \big)$ and  we set
\begin{align}\label{y1}
	\big(\hat{Y}_1^\e(t), \hat{Z}_1^\e(t) \big):= \big(\hat{Y}^\e(t)- P_1(t)X_1(t), \hat{Z}^\e(t) + P_1(t)\delta\sigma(t) 1_{[t_0, t_0+\e]}(t) - Q_1(t)X_1(t)\big).
\end{align}
The following result is the estimate of $\big(\hat{Y}_1^\e(t), \hat{Z}_1^\e(t) \big)$.
\begin{proposition}\label{prop3.9}
	Let \autoref{A1} hold. Then,  we have that
	\begin{align*}
		\dbE^{\dbQ^*}\Big[\mathop{\sup}_{t\in [0, T]} |\hat{Y}_1^\e(t)|^2 + \int_{0}^{T}|\hat{Z}_1^\e(t)|^2dt \Big] = O(\e^{2}).
	\end{align*}
\end{proposition}
\begin{proof}
	Using the definition of $\big(\hat{Y}_1^\e, \hat{Z}_1^\e\big)$ and noting that $$\big(Y^\e(s), Z^\e(s) + P_1(s)\delta \sigma (s)1_{[t_0, t_0+\e]}(s) \big) = \big(\hat{Y}_1^\e(s) + Y^*(s) + P_1(s)X_1(s), \hat{Z}_1^\e(s) + Z^*(s) + Q_1(s)X_1(s) \big),$$ we deduce that
	\begin{equation}\label{tilde-y}
		\begin{aligned}
			\hat{Y}_1^\e(t) =& h\big(X^\e(T)\big) - h\big(X^*(T)\big) - P_1(T)X_1(T) +P_1(T)X_1(T)- P_1(t)X_1(t)  \\
			&- \int_{t}^{T} \Big\{U_1(s)+U_2(s)+U_3(s)\Big\}ds + \int_{t}^{T}\Big\{\hat{Z}_1^\e(s)+Q_1(s)X_1(s) - P_1(s)\delta\sigma(s) 1_{[t_0, t_0+\e]}(s)\Big\}dW(s),
	\end{aligned}\end{equation}
	where 
\begin{align}
	&U_1(s):= f\big(s, X^\e(s), Y^\e(s), Z^\e(s), u^\e(s)\big)\nonumber \\
	&\q\q\q\q- f\big(s, X^*(s)+ X_1(s), Y^\e(s), Z^\e(s) + P_1(s)\delta \sigma (s)1_{[t_0, t_0+\e]}(s), u^*(s)\big),\label{u1}\\
	&U_2(s):= f\big(s, X^*(s)+ X_1(s), \hat{Y}_1^\e(s)+Y^*(s)+ P_1(s)X_1(s),  \hat{Z}_1^\e(s)+ Z^*(s) + Q_1(s)X_1(s), u^*(s)\big) \nonumber\\
	&\q\q\q\q- f\big(s, X^*(s)+ X_1(s), Y^*(s)+ P_1(s)X_1(s), Z^*(s)  + Q_1(s)X_1(s), u^*(s)\big),\nonumber\\
	&U_3(s) := f\big(s, X^*(s)+ X_1(s), Y^*(s)+ P_1(s)X_1(s), Z^*(s)  +Q_1(s)X_1(s), u^*(s)\big) - f(s).\nonumber
\end{align}
	By first-order Taylor's expansion, we obtain that
	\begin{align}\label{combine 0}
		U_2(s) = \tilde{f}^{2\e}_y(s) \hat{Y}_1^\e(s) + \tilde{f}^{2\e}_z(s) \hat{Z}_1^\e(s)=\tilde{f}^{2\e}_y(s) \hat{Y}_1^\e(s) + \big(\tilde{f}^{2\e}_z(s) - f_z(s)\big) \hat{Z}_1^\e(s) + f_z(s)\hat{Z}_1^\e(s),
	\end{align}
	where
	\begin{align*}
		&\tilde{f}^{2\e}_y(s) := \int_{0}^{1}f_y\big(s, X^*(s)+X_1(s), Y^*(s)+P_1(s)X_1(s)+\theta \hat{Y}_1^\e(s), Z^*(s)+ Q_1(s)X_1(s)+ \theta \hat{Z}_1^\e(s), u^*(s)\big)d\theta,\\
		&\tilde{f}^{2\e}_z(s) := \int_{0}^{1}f_z\big(s, X^*(s)+X_1(s), Y^*(s)+P_1(s)X_1(s)+\theta \hat{Y}_1^\e(s), Z^*(s)+ Q_1(s)X_1(s)+ \theta \hat{Z}_1^\e(s), u^*(s)\big)d\theta.
	\end{align*}
	In addition, applying It$\hat{\rm o}$'s formula to $\big(P_1(s)X_1(s)\big)_{s\in[t, T]}$ and first-order Taylor's expansion to $U_3$, we deduce that
	\begin{equation}\label{combine 1}
		\begin{aligned}
			&P_1(T)X_1(T) - P_1(t)X_1(t) - \int_{t}^{T}U_3(s)ds +\int_{t}^{T}\big\{Q_1(s)X_1(s)-P_1(s)\delta \sigma(s)1_{[t_0, t_0+\e]}(s)\big\}dW(s) \\
			&= -\int_{t}^{T}B_1(s)ds,
		\end{aligned}
	\end{equation}
	where 
	\begin{align*}
		&B_1(s) := \big(  \tilde{f}^{3\e}_x(s) - f_x(s) \big)X_1(s)
		+  \big(  \tilde{f}^{3\e}_y(s) -f_y(s) \big)P_1(s)X_1(s) 
		+  \big(\tilde{f}^{3\e}_z(s)-f_z(s)  \big)Q_1(s)X_1(s)  \nonumber \\
		&\q\q\q\q+Q_1(s)\delta \sigma(s)1_{[t_0, t_0+\e]}(s),\label{b1}\\
		&\tilde{f}^{3\e}_x(s) := \int_{0}^{1} f_x\big(s, X^*(s)+ \theta X_1(s), Y^*(s)+ \theta P_1(s)X_1(s), Z^*(s) + \theta Q_1(s)X_1(s), u^*(s)\big)d\theta,
	\end{align*}
	and $\tilde{f}^{3\e}_y(s), \tilde{f}^{3\e}_z(s)$ are defined similarly.
	Combining \eqref{combine 0} and \eqref{combine 1}, we obtain
	\begin{equation}\label{tilde-y-}
		\begin{aligned}
			\hat{Y}_1^\e(t) = H^\e_1(T)  
			- \int_{t}^{T} \Big\{\tilde{f}^{2\e}_y(s) \hat{Y}_1^\e(s) + J_1(s)
			\Big\}ds  
			+ \int_{t}^{T}\hat{Z}_1^\e(s)dW^{q^*}(s),
	\end{aligned}\end{equation}
	where
	\begin{align*}
		H^\e_1(T):= h\big(X^\e(T)\big) - h\big(X^*(T)\big) - P_1(T)X_1(T),\q
		J_1(s):= \big(\tilde{f}^{2\e}_z(s) - f_z(s)\big) \hat{Z}_1^\e(s) +U_1(s)+B_1(s).
	\end{align*}
	It then follows from Briand, Delyon, Hu, Pardoux, and Stoica \cite{briand2003lp} that
	\begin{equation}\label{3.51}
		\begin{aligned}
			&\dbE^{\dbQ^*}\Big[\mathop{\sup}_{s\in [0, T]}|\hat{Y}^\e_1(s)|^2 + \int_{0}^{T}|\hat{Z}^\e_1(s)|^2ds\Big]\\
			&\leq K\dbE^{\dbQ^*}\Big[\big|H^\e_1(T)\big|^2 + \Big(\int_{0}^{T} \big|\tilde{f}^{2\e}_z(s) - f_z(s)\big| |\hat{Z}_1^\e(s)| +|U_1(s)|+|B_1(s)|ds\Big)^2\Big].
	\end{aligned}\end{equation}
	Now we estimate the right side of \eqref{3.51}. 
	
		{\bf Step 1:} we first estimate the term $\dbE^{\dbQ^*}[|H^\e_1(T)|^2]$.
	By Taylor's expansion, \autoref{A1}, and \autoref{state}, we deduce that
	\begin{equation}\label{h1}
		\begin{aligned}
			\dbE^{\dbQ^*}\Big[\big|H^\e_1(T)\big|^2\Big] \leq K\dbE^{\dbQ^*}\Big[\|h_{xx}\|_\i^2|\hat{X}^\e(T)|^4 + \|h_x\|_\i^2|\hat{X}_1^\e(T)|^2\Big] \leq K\e^2.
	\end{aligned}\end{equation}

	{\bf Step 2:} for the term $\dbE^{\dbQ^*}[\big(\int_{0}^{T}\big|\tilde{f}^{2\e}_z(s) - f_z(s)\big|\cd | \hat{Z}_1^\e(s)|ds\big)^2]$,  by \autoref{A1}, Young's inequality, \autoref{wellpose of ad}, \autoref{state}, and  \autoref{importan-},  one has that 
	\begin{equation}\label{c1}
		\begin{aligned}
			&\dbE^{\dbQ^*}\Big[\Big(\int_{0}^{T}\big|\tilde{f}^{2\e}_z(s) - f_z(s)\big| \cd |\hat{Z}_1^\e(s)|ds\Big)^2\Big]\\
			&\leq K\dbE^{\dbQ^*}\Big[\Big(\int_{0}^{T}\big\{|X_1(s)| + |\hat{Y}^\e(s)| +|\hat{Z}^\e(s)| +|N_1(s)| + 1_{[t_0, t_0+\e]}(s) \big\}
			\cd \big\{|\hat{Z}^\e(s)| + |N_1(s)|  \big\}ds\Big)^2\Big]\\
			& \leq K\bigg\{\dbE^{\dbQ^*}\Big[\mathop{\sup}_{s\in[0, T]} |X_1(s)|^4\Big] + \dbE^{\dbQ^*}\Big[\mathop{\sup}_{s\in[0, T]} |\hat{Y}^\e(s)|^4\Big] + \dbE^{\dbQ^*}\Big[\Big(\int_{0}^{T} |\hat{Z}^\e(s)|^2ds\Big)^2\Big] \\
			&\q\q\q+ \dbE^{\dbQ^*}\Big[\Big(\int_{0}^{T} |N_1(s)|^2ds\Big)^2\Big]+  \dbE^{\dbQ^*}\Big[\Big(\int_{t_0}^{t_0+\e} |\hat{Z}^\e(s)|ds\Big)^2\Big] + \dbE^{\dbQ^*}\Big[\Big(\int_{t_0}^{t_0+\e} |N_1(s)|ds\Big)^2\Big]\bigg\} 
			\leq K_{t_0}\e^{2},
		\end{aligned}
	\end{equation}
	where the estimate of $N_1$ is given in \autoref{xx}.\\
	{\bf Step 3:} for the term $\dbE^{\dbQ^*}[\big(\int_{0}^{T}|U_1(s)|ds\big)^2]$, 
	set 
	\begin{align*}
		&U_{11}(s):= f\big(s, X^\e(s), Y^\e(s), Z^\e(s), u^*(s)\big) - f\big(s, X^*(s)+X_1(s), Y^\e(s), Z^\e(s), u^*(s)\big),\\
		&U_{12}(s):= f\big(s, X^\e(s), Y^\e(s), Z^\e(s), u(s)\big) - f\big(s, X^*(s)+X_1(s), Y^\e(s), Z^\e(s)+P_1(s)\delta \sigma(s), u(s)\big),\\
		&U_{13}(s):= f\big(s, X^*(s)+X_1(s), Y^\e(s), Z^\e(s)+P_1(s)\delta \sigma(s), u(s)\big) \\
		&\q\q\q\q- f\big(s, X^*(s)+X_1(s), Y^\e(s), Z^\e(s)+P_1(s)\delta \sigma(s), u^*(s)\big).
	\end{align*}
	One has that
	\begin{align*}
		U_1(s) = U_{11}(s)1_{[0, t_0) \bigcup (t_0+\e, T]}(s) + \big\{U_{12}(s)+U_{13}(s)\big\}1_{[t_0, t_0+\e]}(s).
	\end{align*}
	On the one hand, by \autoref{A1} and \autoref{state}, we obtain that
	\begin{align}\label{com1}
		\dbE^{\dbQ^*}\Big[\Big(\int_{0}^{T}|U_{11}(s)|1_{[0, t_0) \bigcup (t_0+\e, T]}(s) ds\Big)^2\Big]
		\leq K\dbE^{\dbQ^*}\Big[\mathop{\sup}_{s\in [0, T]}|X^\e(s)-X^*(s)-X_1(s)|^2\Big] \leq K\e^2.
	\end{align}
	On the other hand, thanks to \autoref{A1}, \autoref{state},  \autoref{system-wellposedness}, H$\ddot{\rm o}$lder's inequality, \autoref{importan-}, \autoref{q*}, and the Lebesgue differentiation theorem, we deduce that 
	\begin{equation}\label{com2}
		\begin{aligned}
			&\dbE^{\dbQ^*}\Big[\Big(\int_{t_0}^{t_0+\e}|U_{12}(s)|ds\Big)^2\Big] \\
			&\leq K\bigg\{\dbE^{\dbQ^*}\Big[\mathop{\sup}_{s\in [0, T]}|X^\e(s)-X^*(s)-X_1(s)|^2\Big] + \dbE^{\dbQ^*}\Big[\int_{t_0}^{t_0+\e}\big\{1+|Z^\e(s)|^2\big\}ds\Big]\e
			\bigg\}\\
			&\leq K \bigg\{\e^2 + \dbE^{\dbQ^*}\Big[\int_{0}^{T}|\hat{Z}^\e(s)|^2 ds\Big] \e+ \dbE^{\dbQ^*} \Big[\int_{t_0}^{t_0+\e}|Z^*(s)|^2 ds\Big] \e\bigg\} \leq K_{t_0}\e^{2}.
		\end{aligned}
	\end{equation}
	In addition, it follows from \autoref{A1} and \autoref{system-wellposedness} that
	\begin{align}\label{com3}
		\dbE^{\dbQ^*}\Big[\Big(\int_{t_0}^{t_0+\e}|U_{13}(s)|ds\Big)^2\Big] \leq K\Big\{1 + \dbE^{\dbQ^*}\Big[\mathop{\sup}_{s\in [0, T]}|X^*(s)|^2\Big] + \dbE^{\dbQ^*}\Big[\mathop{\sup}_{s\in [0, T]}|X_1(s)|^2\Big]\Big\}\e^2 \leq K\e^2.
	\end{align}
	Combining \eqref{com1}, \eqref{com2}, and \eqref{com3}, we obtain that 
	\begin{align}\label{c2}
		\dbE^{\dbQ^*}\Big[\Big(\int_{0}^{T}|U_1(s)|ds\Big)^2\Big] \leq K\e^{2}.
	\end{align}

	{\bf Step 4:} for the term $\dbE^{\dbQ^*}\big[\big(\int_{0}^{T}|B_1(s)| ds\big)^2\big]$, by Taylor's expansion, \autoref{A1}, Young's inequality, H$\ddot{\rm o}$lder's inequality, \autoref{state}, \autoref{wellpose of ad}, and the Lebesgue differentiation theorem, we deduce that 
	\begin{equation}\label{c3}
		\begin{aligned}
			&\dbE^{\dbQ^*}\Big[\Big(\int_{0}^{T}|B_1(s)| ds\Big)^2\Big]\\
			&\leq K\bigg\{\dbE^{\dbQ^*}\Big[\Big(\int_{0}^{T}\Big\{\Big(|X_1(s)| + |N_1(s)| + 1_{[t_0, t_0+\e]}(s)\Big)^2 + |Q_1(s)|1_{[t_0, t_0+\e]}(s) \Big\}ds\Big)^2\Big] \bigg\}
			\leq K_{t_0}\e^2.
	\end{aligned}\end{equation}

	Combining Steps 1-4, we obtain the desired result. This completes the proof.
\end{proof}

\subsection{The estimate for $Y^\e(0) - Y^*(0) - \hat{Y}(0)$} \label{section6}
In this subsection, we study the estimate of $Y^\e(0) - Y^*(0) - \hat{Y}(0)$. 
Recall that $\hat{Y}$ is defined in \eqref{va} and $\big(\hat{Y}^\e(t), \hat{Z}^\e(t)\big) := \big(Y^\e(t)- Y^*(t), Z^\e(t)- Z^*(t) \big)$. In the sequel, we present the estimate of $Y^\e(0)-Y^*(0)-\hat{Y}(0)$.
\begin{lemma}\label{last}
Let \autoref{A1} hold. Then, we have that
\begin{align}\label{end}
Y^{\e}(0)-Y^*(0) - \hat{Y}(0) = o(\e).
\end{align}
\end{lemma}
\begin{proof}
	Recall that $M_2, \tilde{M}$ and $\tilde{N}_2, \tilde{N}$ are defined in \eqref{22222} and \eqref{33333}, respectively. Set 

	 \begin{align}
	 	&O_1(s) := P_1(s)\big(X_1(s)+X_2(s)\big)+\frac{1}{2}P_2(s)\big(X_1(s)\big)^2 = P_1(s) X_1(s)+ \tilde{M}(s)+\frac{1}{2}M_2(s), \nonumber\\
	 	& O_2(s):=  - Q_1(s)\big(X_1(s)+X_2(s)\big)- \frac{1}{2}Q_2(s)\big(X_1(s)\big)^2+P_2(s)\delta \sigma(s)X_1(s)1_{[t_0, t_0+\e]}(s)\nonumber\\
	 	&\q\q\q
	 	= - Q_1(s)X_1(s) - \tilde{N}(s) - \frac{1}{2}\tilde{N}_2(s)+P_2(s)\delta \sigma(s)X_1(s)1_{[t_0, t_0+\e]}(s),\nonumber\\
	 	&\hat{Y}_2^\e(s) :=\hat{Y}^\e(s) - O_1(s), \q \hat{Z}_2^\e(s):=\hat{Z}^\e(s)  + P_1(s)\delta\sigma(s) 1_{[t_0, t_0+\e]}(s) +O_2(s),
	 	\nonumber\\
	 &\hat{\mathcal{Y}}_2^\e(s) := \hat{Y}_2^\e(s)-\hat{Y}(s)   ,\q \hat{\mathcal{Z}}_2^\e(s):=\hat{Z}_2^\e(s) - \hat{Z}(s). \label{o}
	 \end{align}
It then follows that
 \begin{equation}\label{tilde-y2}
 	\begin{aligned}
 		\hat{\mathcal{Y}}_2^\e(t) =& H^\e_2(T) +O_1(T)- O_1(t)  - \int_{t}^{T} \Big\{V_1(s)+V_2(s)+V_3(s)-f_y(s)\hat{Y}(s) - f_z(s)\hat{Z}(s)-I(s)\Big\}ds\\
 		& + \int_{t}^{T}\Big\{\hat{\mathcal{Z}}_2^\e(s)- P_1(s)\delta\sigma(s) 1_{[t_0, t_0+\e]}(s) - O_2(s)\Big\}dW(s),
 \end{aligned}\end{equation}
 where $I(s)$ is defined in \eqref{i}, and
 \begin{align*}
 	&H^\e_2(T):= h\big(X^\e(T)\big) - h\big(X^*(T)\big) - O_1(T),\\
 	&V_1(s):= f\big(s, X^\e(s), Y^\e(s), Z^\e(s), u^\e(s)\big) \\
 	&\q\q\q\q- f\big(s, X^*(s)+ X_1(s)+X_2(s), Y^*(s)+\hat{Y}_2^\e(s)+ O_1(s), Z^*(s) + \hat{Z}_2^\e(s) - O_2(s), u^*(s)\big),\\
 	&V_2(s):= f\big(s, X^*(s)+ X_1(s)+X_2(s), Y^*(s)+\hat{Y}_2^\e(s)+ O_1(s), Z^*(s) + \hat{Z}_2^\e(s) - O_2(s), u^*(s)\big) \\
 	&\q\q\q\q- f\big(s, X^*(s)+ X_1(s)+X_2(s), Y^*(s)+ O_1(s), Z^*(s)  -O_2(s), u^*(s)\big),\\
 	&V_3(s) := f\big(s, X^*(s)+ X_1(s)+X_2(s), Y^*(s)+ O_1(s), Z^*(s)  -O_2(s), u^*(s)\big) - f(s).
 \end{align*}
	Applying It$\hat{\rm o}$'s formula to $\big(O_1(s)\big)_{s\in[t, T]}$, first-order Taylor's expansion to $V_2$,  and second-order Taylor's expansion to $h, V_3$, and arguing as in the derivation of \eqref{tilde-y-}, one can deduce that
	\begin{equation}
	\begin{aligned}
	\hat{\mathcal{Y}}_2^\e(t) &= H^\e_2(T)  
	- \int_{t}^{T} \Big\{f^{2\e}_y(s) \hat{Y}_2^\e(s) - f_y(s)\hat{Y}(s) + f^{2\e}_z(s) \hat{Z}_2^\e(s) - f_z(s)\hat{Z}(s) \\
	& \q\q\q\q\q\q\q+ V_1(s) + J_3(s) + J_4(s) + J_5(s)
	\Big\}ds  
	+ \int_{t}^{T}\hat{\mathcal{Z}}_2^\e(s)dW(s)\\
	&= H^\e_2(T)  
	- \int_{t}^{T} \Big\{f_y(s)\hat{\mathcal{Y}}_2^\e(s) + J_2(s)+ V_1(s)+ J_3(s) + J_4(s) + J_5(s)\Big\}ds  
	+ \int_{t}^{T}\hat{\mathcal{Z}}_2^\e(s)dW^{q^*}(s),
	\end{aligned}\end{equation}
	where 
	\begin{align*}
	&J_2(s):=\big(f^{2\e}_y(s) - f_y(s)\big)\hat{Y}^\e_2(s) + \big(f^{2\e}_z(s) - f_z(s)\big)\hat{Z}^\e_2(s),
	\\
	&J_3(s):= -\big[f\big(s, X^*(s), Y^*(s), Z^*(s)-P_1(s)\delta \sigma(s), u(s) \big) - f(s)\big]1_{[t_0, t_0+\e]}(s),
	\\
	&J_4(s):= \frac{1}{2}[X_1(s), P_1(s)X_1(s), \tilde{N}_1(s)]\big(D^{2\e}f(s) - D^2 f(s)\big)[X_1(s), P_1(s)X_1(s), \tilde{N}_1(s)]^\top,
	%
	%
	\\
	&J_5(s):= -f_z(s)P_2(s)\delta \sigma(s)X_1(s)1_{[t_0, t_0+\e]}(s) + Q_2(s)\delta \sigma(s)X_1(s)1_{[t_0, t_0+\e]}(s) + f^\e_{xx}(s)X_1(s)X_2(s)\\
	&\q\q\q\q+\frac{1}{2}f_{xx}^\e(s)\big(X_2(s)\big)^2 + \frac{1}{2} f_{yy}^\e(s) \tilde{O}_1(s)+ \frac{1}{2}f_{zz}^\e(s)\tilde{O}_2(s) + f^\e_{xy}(s)\tilde{O}_3(s)- f^\e_{xz}(s)\tilde{O}_4(s) -f^\e_{yz}(s)\tilde{O}_5(s) ,
	\\
	&\tilde{O}_1(s):=\big(O_1(s)\big)^2-\big(P_1(s)X_1(s)\big)^2 = 2P_1(s)X_1(s)\big(\tilde{M}(s)+\frac{1}{2}M_2(s)\big) + \big(\tilde{M}(s)+\frac{1}{2}M_2(s)\big)^2,\\
	&\tilde{O}_2(s):= \big(O_2(s)\big)^2-\big(Q_1(s)X_1(s)\big)^2= 2Q_1(s)X_1(s)\big\{\tilde{N}(s) + \frac{1}{2}\tilde{N}_2(s) - P_2(s)\delta\sigma(s) X_1(s)1_{[t_0, t_0+\e]}(s)\big\} \\
	&\q\q\q\q
	+ \big\{\tilde{N}(s) + \frac{1}{2}\tilde{N}_2(s) - P_2(s)\delta\sigma(s) X_1(s)1_{[t_0, t_0+\e]}(s)\big\}^2, \\
	& \tilde{O}_3(s):=O_1(s)\big(X_1(s)+X_2(s)\big)-P_1(s)\big(X_1(s)\big)^2 \\
	&\q\q\q= P_1(s)X_1(s)X_2(s)+\big(\tilde{M}(s)+\frac{1}{2}M_2(s)\big)\big(X_1(s)+X_2(s)\big), \\
	 &\tilde{O}_4(s):=O_2(s)\big(X_1(s)+X_2(s)\big) + Q_1(s)\big(X_1(s)\big)^2\\
	 &\q\q\q
	 =-Q_1(s)X_1(s)X_2(s) - \big\{\tilde{N}(s) + \frac{1}{2}\tilde{N}_2(s) - P_2(s)\delta\sigma(s) X_1(s)1_{[t_0, t_0+\e]}(s)\big\}\big(X_1(s)+X_2(s)\big),
	\\
	&\tilde{O}_5(s):=O_1(s)O_2(s)+P_1(s)Q_1(s)\big(X_1(s)\big)^2\\
	&\q\q\q
	= P_1(s)X_1(s) \big\{-\tilde{N}(s) - \frac{1}{2}\tilde{N}_2(s) + P_2(s)\delta\sigma(s) X_1(s)1_{[t_0, t_0+\e]}(s)\big\}+ \big(\tilde{M}(s)+\frac{1}{2}M_2(s)\big)O_2(s),\\
	&f_y^{2\e}(s):= \int_{0}^{1}f_y\big(s, X^*(s)+X_1(s)+X_2(s), Y^*(s)+O_1(s)+\theta \hat{Y}^\e_2(s), Z^*(s)-O_2(s)+\theta \hat{Z}^\e_2(s), u^*(s)\big)d\theta,
	 \end{align*}
    \begin{equation}
    	\begin{aligned}
    		&D^{2\e}f(s) :=
    		\begin{bmatrix}
    			f^\e_{xx}(s) & f^\e_{ xy}(s) & f^\e_{xz}(s)\\ 
    			f^\e_{xy}(s) & f^\e_{ yy}(s) & f^\e_{yz}(s)\\
    		f^\e_{xz}(s) & f^\e_{yz}(s) & f^\e_{zz}(s)
    		\end{bmatrix},\,\,\,\,\,\,\,\,\,\,\,\,\,\,\,\,\,\,\,\,\,\,\,\,\,\,\,\,\,\,\,\,\,\,\,\,\,\,\,\,\,\,\,\,\,\,\,\,\,\,\,\,\,\,\,\,\,\,\,\,\,\,\,\,\,\,\,\,\,\,\,\,\,\,\,\,\,\,\,\,\,\,\,\,\,\,\,\,\,\,\,\,\,\,\,\,\,\,\,\,\,\,\,\,\,\,\,\,\,\,\,\,\,\,\,\,\,\,\,\,\,\,\,\,\,\,\,\,\,\,\,\,\,\,\,\,\,\,\,\,\,\,\,\,\,\,\,\,\,\,\,\,\,\,\,\,\,\,\,\,\,\,\,\,\,\,\,\,\,\,\,\,\,\,\,\,\,\,\,\,\,\,\,\,\,\,\,\,\,\\
    		&f_{xx}^\e(s):=2\int_{0}^{1} (1-\theta) f_{xx}\big(s, X^*(s)+\theta\big(X_1(s)+X_2(s)\big), Y^*(s)+\theta O_1(s), Z^*(s)-\theta O_2(s), u^*(s)\big)d\theta, \,\, \,\,\,\,\,\,\,\, \,\,\,\,\,\,\,\, \,\,\,\,\,\,\,\, \,\,\,\,\,\,\,\, \,\,\,\,\,\,\,\, \,\,\,\,\,\,\,\, \,\,\,\,\,\,\,\, \,\,\,\,\,\,\,\, \,\,\,\,\,\,\,\,\,\,\,\,\,\,\,\,\,\,\,\,\,\,\,\,\,\,\,\,\,\,\,\,\,\,\,\,\,\,\,\,\,\,\,\,   
    	\end{aligned}
    \end{equation}
and $f^{2\e}_{z}(s), f_{yy}^{\e}(s),f_{zz}^\e(s),f_{xy}^\e(s),f_{xz}^\e(s),f_{yz}^\e(s)$ are defined similarly.
Furthermore, define 
\begin{align}\label{ga}
\Gamma(t):= \exp\Big\{-\int_{0}^{t}f_y(s)ds\Big\}.
\end{align}
 Then,
applying It$\hat{\rm o}$'s formula to $\big\{\Gamma(t)\hat{\mathcal{Y}}_2^\e(t)\big\}_{t\in[0, T]}$, taking expectation $\dbE^{\dbQ^*}[\cd]$, and using \autoref{A1}, we obtain that
\begin{align}\label{aim}
|\hat{\mathcal{Y}}^\e_2(0)|&\leq \dbE^{\dbQ^*}\Big[\Big|\Gamma(T)H^\e_2(T) - \int_{0}^{T}\Gamma(s)\Big\{J_2(s)+V_1(s)+J_3(s)+J_4(s)+J_5(s)\Big\}ds\Big|\Big]\nonumber\\
&\leq K \dbE^{\dbQ^*}\Big[|H^\e_2(T)| + \int_{0}^{T}\Big\{|J_2(s)|+|V_1(s)+J_3(s)|+|J_4(s)|+|J_5(s)|\Big\}ds\Big].
\end{align}

{\bf Step 1:} we estimate the term $\dbE^{\dbQ^*}[|H^\e_2(T)|]$. By second-order Taylor's expansion, \autoref{A1}, \autoref{state}, and the dominated convergence theorem, one can deduce that
	\begin{equation}\label{co1}
\begin{aligned}
&\dbE^{\dbQ^*}\big[|H^\e_2(T)|\big]\\
&= \dbE^{\dbQ^*}\Big[\Big|h_x\big(X^*(T)\big)\big(\hat{X}^\e(T)- X_1(T)-X_2(T)\big) + \Big(\int_{0}^{1}(1-\theta)\{h_{xx}(X^*(T)+\theta \hat{X}^\e(T)) - h_{xx}(X^*(T))\}d\theta\Big)\\
&\q\q\q\q
\cd\big(\hat{X}^\e(T)\big)^2 + \frac{1}{2} h_{xx}(X^*(T))\big(\hat{X}^\e(T)+X_1(T)\big)\big(\hat{X}^\e(T)-X_1(T)\big)\Big|\Big] = o(\e).
\end{aligned}\end{equation}

{\bf Step 2:} we estimate the term $\dbE^{\dbQ^*}[\int_{0}^{T}|J_2(s)|ds]$. 
On the one hand, recall that $(\hat{Y}^\e_1, \hat{Z}^\e_1)$ is defined in \eqref{y1}. Then, one can check that 
\begin{align*}
&\hat{Y}^\e_2(s)= \hat{Y}^\e_1 (s)- P_1(s)X_2(s) - \frac{1}{2}P_2(s)\big(X_1(s)\big)^2, \\
&
 \hat{Z}^\e_2(s)= \hat{Z}^\e_1 (s) - Q_1(s)X_2(s) - \frac{1}{2}Q_2(s)\big(X_1(s)\big)^2 + P_2(s)\delta \sigma(s)X_1(s)1_{[t_0, t_0+\e]}(s)\\
 &\q\q= \hat{Z}^\e_1 (s) - \tilde{N}(s) - \frac{1}{2}\tilde{N}_2(s) + P_2(s)\delta \sigma(s)X_1(s)1_{[t_0, t_0+\e]}(s) .
\end{align*}
Thanks to \autoref{prop3.9}, \autoref{A1}, H$\ddot{\rm o}$lder's inequality, \autoref{xx}, and \autoref{wellpose of ad}, we obtain that  
\begin{align}\label{y2}
	\dbE^{\dbQ^*}\Big[\mathop{\sup}_{s\in [0, T]} |\hat{Y}_2^\e(s)|^2 + \int_{0}^{T}|\hat{Z}_2^\e(s)|^2ds \Big] = O(\e^{2}).
\end{align}
On the other hand, by first-order Taylor's expansion, one can check that
\begin{align}
&|J_2(s)| \leq K\Big\{ |X_1(s)| + |X_2(s)|+|O_1(s)| + |O_2(s)| + |\hat{Y}_2^\e(s)| + |\hat{Z}_2^\e(s)| \Big\}\cd \big[|\hat{Y}_2^\e(s)| + |\hat{Z}_2^\e(s)| \big]\nonumber\\
&
\leq K\Big\{ |X_1(s)| + |X_2(s)|+|M_1(s)| + |\tilde{N}_1(s)| + |\tilde{M}(s)|+|\tilde{N}(s)| +  |M_2(s)|+|\tilde{N}_2(s)| \nonumber\\
&\q\q\q+ |P_2(s)||X_1(s)|1_{[t_0, t_0+\e]}(s)+|\hat{Y}_2^\e(s)| + |\hat{Z}_2^\e(s)| \Big\}\cd \big[|\hat{Y}_2^\e(s)|  + |\hat{Z}_2^\e(s)| \big].\label{j2}
\end{align}
It then follows from  H$\ddot{\rm o}$lder's inequality, \autoref{state}, \autoref{wellpose of ad}, \autoref{xx}, and \eqref{y2}  that 
\begin{align}\label{co2}
\dbE^{\dbQ^*}\Big[\int_{0}^{T}|J_2(s)|ds\Big] = o(\e).
\end{align}

{\bf Step 3:} we estimate the term $\dbE^{\dbQ^*}\big[\int_{0}^{T}|V_1(s)+J_3(s)|ds\big]$. It is easy to check that
\begin{align*}
Y^*(s)+ \hat{Y}^\e_2(s) + O_1(s) = Y^\e(s),  \q Z^*(s)+ \hat{Z}^\e_2(s) - O_2(s) = Z^\e(s) + P_1(s)\delta\sigma(s)1_{[t_0,t_0+\e]}(s).
\end{align*}
Then, we deduce that
\begin{align*}
V_1(s)+J_3(s)= D_1(s) 1_{[0, t_0)\bigcup(t_0+\e, T]}(s) + D_2(s)1_{[t_0, t_0+\e]}(s) - D_3(s)1_{[t_0, t_0+\e]}(s), 
\end{align*}
where 
\begin{align*}
&D_1(s):= f\big(s, X^\e(s), Y^\e(s), Z^\e(s), u^*(s)\big) - f\big(s, X^*(s)+X_1(s)+X_2(s), Y^\e(s), Z^\e(s), u^*(s)\big), \\
& D_2(s):= f\big(s, X^\e(s), Y^\e(s), Z^\e(s), u(s)\big) - f\big(s, X^*(s), Y^*(s), Z^*(s) - P_1(s)\delta \sigma(s)1_{[t_0, t_0+\e]}(s), u(s)\big),\\
&
D_3(s):= f\big(s, X^*(s)+X_1(s)+X_2(s), Y^\e(s), Z^\e(s)+P_1(s)\delta\sigma(s)1_{[t_0, t_0+\e]}(s), u^*(s)\big) - f(s).
\end{align*}
%
%
We can check that $\hat{Z}^\e(s) + P_1(s)\delta \sigma(s)1_{[t_0, t_0+\e]}(s) = \hat{Z}^\e_1(s) + Q_1(s) X_1(s)$ and recall that $\hat{X}^\e_2(s)= \hat{X}^\e(s)-X_1(s)-X_2(s)$.
It then follows from \autoref{A1} that
\begin{equation}
\begin{aligned}\label{co3}
&\dbE^{\dbQ^*}\Big[\int_{0}^{T}|V_1(s)+J_3(s)|ds\Big]\\ 
&
\leq K\dbE^{\dbQ^*}\Big[\int_{0}^{T}\Big\{|\hat{X}_2^\e(s) |+\Big(|\hat{X}^\e(s)| + |X_1(s)|+|X_2(s)|+ |\hat{Y}^\e(s)| +  \big[|\hat{Z}_1^\e (s)|+|Q_1(s)X_1(s)|\big]\\
&\q\q\q\q\q\q\q\q\q\q\q\cd\big[ |Z^*(s)|+|\hat{Z}_1^\e(s)|+|Q_1(s)X_1(s)| + 1_{[t_0, t_0+\e]}(s)\big] \Big)1_{[t_0, t_0+\e]}(s)\Big\}ds \Big].
\end{aligned}\end{equation}
By H$\ddot{\rm o}$lder's inequality, \autoref{system-wellposedness}, \autoref{q*}, \autoref{prop3.9}, and the Lebesgue differentiation theorem, we deduce that
\begin{align*}
&\dbE^{\dbQ^*}\Big[\int_{t_0}^{t_0+\e}|Z^*(s)|\cd |\hat{Z}_1^\e(s)|ds\Big] \leq \dbE^{\dbQ^*}\Big[\int_{t_0}^{t_0+\e}|Z^*(s)|^2ds\Big]^{\frac{1}{2}} \dbE^{\dbQ^*}\Big[\int_{t_0}^{t_0+\e}|\hat{Z}^\e_1(s)|^2ds\Big]^{\frac{1}{2}} \leq K_{t_0}\e^\frac{3}{2}.
\end{align*}
Similarly, we deduce that
\begin{align*}
	&\dbE^{\dbQ^*}\Big[\int_{t_0}^{t_0+\e}|Z^*(s)Q_1(s)X_1(s)|ds\Big] \\
	&\leq \dbE^{\dbQ^*}\Big[\int_{t_0}^{t_0+\e}|Z^*(s)|^2ds\Big]^{\frac{1}{2}}\dbE^{\dbQ^*}\Big[\Big(\int_{0}^{T}|Q_1(s)|^21_{[t_0, t_0+\e]}(s)ds\Big)^2\Big]^{\frac{1}{4}}  \dbE^{\dbQ^*}\Big[\mathop{\sup}_{s\in [0, T]}|X_1(s)|^4\Big]^{\frac{1}{4}} \\
	&
	\leq K_{t_0}\dbE^{\dbQ^*}\Big[\Big(\int_{0}^{T}|Q_1(s)|^21_{[t_0, t_0+\e]}(s)ds\Big)^2\Big]^{\frac{1}{4}} \e.
\end{align*}
By the dominated convergence theorem, one can deduce that
\begin{align*}
\lim\limits_{\e \rightarrow 0}\dbE^{\dbQ^*}\Big[\Big(\int_{0}^{T}|Q_1(s)|^21_{[t_0, t_0+\e]}(s)ds\Big)^2\Big]^{\frac{1}{4}} = 0.
\end{align*}
Thus, we obtain that 
$\dbE^{\dbQ^*}\Big[\int_{t_0}^{t_0+\e}|Z^*(s)Q_1(s)X_1(s)|ds\Big]  = o(\e).$
The other terms on the right-hand side of \eqref{co3} can be estimated similarly.
Therefore, we obtain that
\begin{align*}
\dbE^{\dbQ^*}\Big[\int_{0}^{T}|V_1(s)+J_3(s)|ds\Big] = o(\e).
\end{align*}

{\bf Step 4:} for the term $\dbE^{\dbQ^*}\big[\int_{0}^{T}|J_4(s)|ds\big]$, thanks to H$\ddot{\rm o}$lder's inequality, \autoref{A1}, \autoref{state},  \autoref{xx}, and the dominated convergence theorem, we deduce that
\begin{align*}
&\limsup\limits_{\e \rightarrow 0}\frac{1}{\e}\dbE^{\dbQ^*}\Big[\int_{0}^{T}|J_4(s)|ds\Big] \\
&\leq K\limsup\limits_{\e \rightarrow 0}\dbE^{\dbQ^*}\Big[\mathop{\sup}_{s\in[0, T]}\frac{|X_1(s)|^4}{\e^2}\Big]^\frac{1}{2} \dbE^{\dbQ^*}\Big[\Big(\int_{0}^{T}|D^{2\e}f(s) - D^2f(s)|(1+|Q_1(s)|^2)ds\Big)^2\Big]^\frac{1}{2} \\
&\leq K\limsup\limits_{\e \rightarrow 0} \dbE^{\dbQ^*}\Big[\Big(\int_{0}^{T}|D^{2\e}f(s) - D^2f(s)|(1+|Q_1(s)|^2)ds\Big)^2\Big]^\frac{1}{2}  = 0.
\end{align*}
Thus, we obtain that
$
	\dbE^{\dbQ^*}\Big[\int_{0}^{T}|J_4(s)|ds\Big] = o(\e).
$

{\bf Step 5:} we estimate the term $\dbE^{\dbQ^*}\big[\int_{0}^{T}|J_5(s)|ds\big]$.
%
%
On the one hand, by H$\ddot{\rm o}$lder's inequality, \autoref{state}, and \autoref{q*}, we obtain that
\begin{align*}
	&\dbE^{\dbQ^*}\Big[\int_{t_0}^{t_0+\e}|f_z(s)P_2(s)\delta \sigma(s)X_1(s)|ds\Big]\leq K\dbE^{\dbQ^*}\Big[\mathop{\sup}_{s\in[0, T]}|X_1(s)|^2\Big]^{\frac{1}{2}}\dbE^{\dbQ^*}\Big[\Big(\int_{t_0}^{t_0+\e}|f_z(s)|ds\Big)^2\Big] ^\frac{1}{2} \\
	&
	\leq K  \dbE^{\dbQ^*}\Big[\mathop{\sup}_{s\in[0, T]}|X_1(s)|^2\Big]^\frac{1}{2} \cd \dbE^{\dbQ^*}\Big[\int_{t_0}^{t_0+\e}|f_z(s)|^2ds\Big]^\frac{1}{2}\e^{\frac{1}{2}}  \leq K_{t_0} \e^{\frac{3}{2}},
\end{align*}
and by a similar deduction, we obtain that
\begin{align*}
\dbE^{\dbQ^*}\Big[\int_{t_0}^{t_0+\e}|Q_2(s)\delta \sigma(s)X_1(s)|ds\big] \leq K_{t_0}\e^{\frac{3}{2}}.
\end{align*}
On the other hand, 
by \autoref{xx}, we deduce that 
\begin{equation}\label{oo}
\begin{aligned}
&\dbE^{\dbQ^*}\Big[\int_{0}^{T}|-\tilde{N}(s) - \frac{1}{2}\tilde{N}_2(s) + P_2(s)\delta\sigma(s) X_1(s)1_{[t_0, t_0+\e]}(s)|^2ds\Big] = O(\e^2),\\ 
 &\dbE^{\dbQ^*}\Big[\int_{0}^{T}|O_2(s)|^2ds\Big] = O(\e).
\end{aligned}\end{equation}
It then follows from H$\ddot{\rm o}$lder's inequality, \autoref{state}, \autoref{xx}, and \eqref{oo} that 
\begin{align*}
\dbE^{\dbQ^*}\Big[\int_{0}^{T}|f_{yz}^\e(s)\tilde{O}_5(s)|ds\Big] &\leq K\bigg\{\dbE^{\dbQ^*}\Big[\mathop{\sup}_{s\in[0, T]}|X_1(s)|^2\Big]^{\frac{1}{2}}\dbE^{\dbQ^*}\Big[\int_{0}^{T}|O_2(s)+Q_1(s)X_1(s)|^2ds\Big]^{\frac{1}{2}} \\
&\q\q\q+ \dbE^{\dbQ^*}\Big[\int_{0}^{T}|O_2(s)|^2ds\Big]^{\frac{1}{2}}\dbE^{\dbQ^*}\Big[\int_{0}^{T}|\tilde{M}(s)+\frac{1}{2}M_2(s)|^2ds\Big]^{\frac{1}{2}}\bigg\} \leq K\e^\frac{3}{2}.
\end{align*}
For the other terms of $\dbE^{\dbQ^*}\big[\int_{0}^{T}|J_5(s)|ds\big]$, we can estimate them similarly.
Consequently, we obtain that
\begin{align}
	\dbE^{\dbQ^*}\Big[\int_{0}^{T}|J_5(s)|ds\Big] = o(\e).
\end{align}

Therefore, combining Steps 1-5, we obtain \eqref{end}.

\end{proof}
\section{Maximum principle}\label{section7}
Based on the above preparation, now we can state the general maximum principle. Define the Hamiltonian function: $\mathcal{H}: [0, T] \times \Om \times \dbR^6 \times U \rightarrow \dbR$ by
\begin{align*}
	\mathcal{H}(t,x, y,z,P_1,Q_1, P_2, u) :=&  P_1 b(t, x, u)-  Q_1 \sigma(t, u) - f\big(t, x, y, z- P_1 \big(\sigma (t, u) - \sigma(t, u^*(t))\big), u\big) \\
	& + \frac{1}{2}P_2\Big(\sigma(t,  u) - \sigma(t, u^*(t))\Big)^2,
\end{align*}
where $P_1, Q_1$ and $ P_2$ are defined in \eqref{first ad} and \eqref{second ad}, respectively.
Now we present the stochastic maximum principle. Recall that $t_0$ is the left endpoint for the small interval of perturbation and $\Gamma, I$ are defined in \eqref{ga} and \eqref{i}, respectively.
\begin{theorem}\label{SMP-thm}
	Suppose that \autoref{A1}  holds. Let $u^*(\cd) \in \mathcal{U}_{ad}$ be an optimal control and $(X^*(\cd), Y^*(\cd), Z^*(\cd))$ be the corresponding state trajectories of \eqref{system}. Then, for each $u\in U$, the following stochastic maximum principle holds:
	\begin{align}\label{SMP}
		\Delta\mathcal H(t,u)\geq 0,
		\qquad dt\otimes d\mathbb P\text{-a.e.},
	\end{align}
	where
	\begin{align*}
		\Delta \mathcal{H}(t, u):=&\mathcal{H}\big(t, X^*(t), Y^*(t), Z^*(t), P_1(t), Q_1(t), P_2(t), u\big)\\
		&- \mathcal{H}\big(t, X^*(t), Y^*(t), Z^*(t), P_1(t), Q_1(t), P_2(t), u^*(t)\big).
	\end{align*}
\end{theorem}

\begin{proof}
	In view of \eqref{va} and by \autoref{last}, we obtain that 
	\begin{equation}\label{qu}
		\begin{aligned}
			0 \leq &J(u^{\e}(\cd)) - J(u^*(\cd)) = Y^{\e}(0) - Y^*(0) = \hat{Y}(0) + o(\e) = -\dbE^{\dbQ^*}\Big[\int_{0}^{T}\Gamma (t)I(t)dt\Big]+ o(\e) \\
			&= \dbE^{\dbQ^*}\Big[\int_{t_0}^{t_0+\e}\Gamma(t)\Delta \mathcal{H}\big(t, u(t)\big) dt\Big] + o(\e).
	\end{aligned}	\end{equation}
		Then, we have that 
		\begin{align*}
			0 \leq \frac{1}{\e}\int_{t_0}^{t_0+\e}\dbE^{\dbQ^*}\big[\Gamma(t)\Delta \mathcal{H}\big(t, u(t)\big)\big]  dt+ \frac{o(\e)}{\e}.
			\end{align*}
			It then follows from the Lebesgue differentiation theorem that 
			\begin{align*}  
					   0 \leq \lim\limits_{\e \rightarrow 0}\frac{1}{\e}\int_{t_0}^{t_0+\e}\dbE^{\dbQ^*}\big[\Gamma(t)\Delta \mathcal{H}\big(t, u(t)\big)\big]  dt+  \lim\limits_{\e \rightarrow 0}\frac{o(\e)}{\e} = \dbE^{\dbQ^*}\big[\Gamma(t_0)\Delta \mathcal{H}\big(t_0, u(t_0)\big)\big] .
				\end{align*}
		Thus, we obtain that $\dbE^{\dbQ^*}[\Gamma(t)\Delta \mathcal{H}\big(t, u(t) \big)] \geq 0, \forall u(\cd)\in \mathcal{U}_{ad}, a.e.$.
		Furthermore, set 
		\begin{align*}
		C^u:= \big\{(t, \omega)\in [0, T] \times \Omega\big| \Delta \mathcal{H}(t, u)(\omega)<0\big\},\q
		\tilde{u}(t)= u1_{C^u}  + u^*(t) 1_{[0, T] \times \Omega \backslash C^u}.
		\end{align*}
Since $C^u$ is progressively measurable, $\tilde{u}(\cd) \in \mathcal{U}_{ad}$. Then, we deduce that
	\begin{align}\label{ee}
		 0\leq \dbE^{\dbQ^*}[\Gamma(t)\Delta \mathcal{H}\big(t, \tilde{u}(t)\big)1_{C^u} ]  =   \dbE^{\dbQ^*}[\Gamma(t)\Delta \mathcal{H}(t, u )1_{C^u}] \leq 0, \q a.e.,
	\end{align}
	which implies \eqref{SMP}.
	\end{proof}
If the control domain $U$ is convex, we can relax the boundedness assumption on $\sigma$.
The admissible controls associated with the convex control domain are given by
\begin{equation}\label{u}
	\begin{aligned}
		\mathcal{\widetilde{U}}^\beta_{ad} := &\Big\{u(\cd) : [0, T] \rightarrow U \ \big|\  \hb{$u(\cd)$  is a   $\dbF$-progressively measurable process} ,\\
		&\hb{and} \q \dbE\Big[\exp\Big\{\beta\int_{0}^{T}|u(s)|^2ds\Big\}\Big] < \i   \Big\}.
\end{aligned}\end{equation}
\begin{remark}
	We give an example in which the integrability condition in \eqref{u} is equivalent to that in Delarue and Lavigne \cite[inequality (4)]{delarue2026robust}. The following notation in this remark is taken from Delarue and Lavigne \cite{delarue2026robust}.
	Let
	\begin{align*}
		f\big(t, Y(t), Z(t)\big) = \frac{\gamma \beta}{2} \big(Z(t)\big)^2.
	\end{align*}
	In view of the definition of $\mathcal{S}^*(\psi)$ and the fact that 
	\begin{align*}
		q(t) = \dbE_t[q(T)],\q q(T) = \exp\Big\{\int_{0}^{T}Z^\star(s)dW(s) - \frac{1}{2}\int_{0}^{T}|Z^\star(s)|^2ds\Big\},
	\end{align*}
	we deduce that
	\begin{align*}
		&\mathcal{S}^*(\psi) = \mathop{\sup}_{q\in\mathcal{Q}}\bigg\{\dbE\Big[\int_{0}^{T}q(s)|\psi(s)|^2ds\Big] - \gamma \dbE\Big[\int_{0}^{T}q(s)f^\star\big(s, Y^\star(s), Z^\star(s)\big)ds\Big]\bigg\}\\
		&
		= \mathop{\sup}_{\dbQ\in\mathscr{Q}}\bigg\{\dbE^{\dbQ}\Big[\int_{0}^{T}|\psi(s)|^2ds\Big] -  \dbE^\dbQ\Big[\int_{0}^{T}\frac{1}{2\beta}|Z^\star(s)|^2ds\Big]\bigg\}\\
		&
		=\mathop{\sup}_{\dbQ\in\mathscr{Q}}\bigg\{\dbE^{\dbQ}\Big[\int_{0}^{T}|\psi(s)|^2ds\Big] -  \frac{1}{\beta}\dbE^\dbQ\Big[\log\frac{d\dbQ}{d\dbP} \Big]\bigg\},
	\end{align*}
	where 
	\begin{align*}
		\mathscr{Q}:= \Big\{\dbQ\Big| \dbQ \sim \dbP \q\hb{and}\q \dbE\Big[\frac{d\dbQ}{d\dbP}\log \frac{d\dbQ}{d\dbP}\Big] <\i\Big\}.
	\end{align*}
	{\bf Step 1:} if $\int_{0}^{T}|\psi(s)|^2ds$ is bounded, define 
	\begin{align*}
		\frac{d\tilde{\dbQ}}{d\dbP} := \frac{e^{\int_{0}^{T}\beta|\psi(s)|^2ds}}{\dbE\big[e^{\int_{0}^{T}\beta|\psi(s)|^2ds}\big]}.
	\end{align*}
	Then, we deduce that
	\begin{align*}
		\dbE^\dbQ\Big[\log\frac{d\dbQ}{d\tilde{\dbQ}} \Big] =  \dbE^\dbQ\Big[\log\frac{d\dbQ}{d\dbP} +\log\frac{d\dbP}{d\tilde{\dbQ}} \Big] =  \dbE^\dbQ\Big[\log\frac{d\dbQ}{d\dbP}\Big] - \beta \dbE^\dbQ\Big[\int_{0}^{T}|\psi(s)|^2ds\Big] + \log\dbE\big[e^{\int_{0}^{T}\beta|\psi(s)|^2ds}\big].
	\end{align*}
	It then follows that
	\begin{align*}
		\dbE^{\dbQ}\Big[\int_{0}^{T}|\psi(s)|^2ds\Big] - \frac{1}{\beta} \dbE^\dbQ\Big[\log\frac{d\dbQ}{d\dbP} \Big] = \frac{1}{\beta} \log\dbE\big[e^{\int_{0}^{T}\beta|\psi(s)|^2ds}\big] - \frac{1}{\beta} \dbE^\dbQ\Big[\log\frac{d\dbQ}{d\tilde{\dbQ}} \Big] \leq \frac{1}{\beta} \log\dbE\big[e^{\int_{0}^{T}\beta|\psi(s)|^2ds}\big],
	\end{align*}
	where the equality holds when $\dbQ=\tilde{\dbQ}$.
	Therefore, we obtain that
	\begin{align*}
		\mathcal{S}^*(\psi) = \frac{1}{\beta}\log\dbE\big[e^{\int_{0}^{T}\beta|\psi(s)|^2ds}\big].
	\end{align*}
	{\bf Step 2:}
	if $\int_{0}^{T}|\psi(s)|^2ds$ is unbounded, define
	\begin{align*}
		V^n := \int_{0}^{T}|\psi(s)|^2ds \wedge n,\q \frac{d\tilde{\dbQ}^n}{d\dbP} := \frac{e^{\beta V^n}}{\dbE\big[e^{\beta V^n}\big]},
	\end{align*}
	by an argument similar to that step 1, we obtain that
	\begin{align*}
	\mathop{\sup}_{\dbQ\in\mathscr{Q}}\bigg\{\dbE^{\dbQ}\big[V^n\big] -  \frac{1}{\beta}\dbE^\dbQ\Big[\log\frac{d\dbQ}{d\dbP} \Big]\bigg\} = \frac{1}{\beta}\log\dbE\big[e^{\beta V^n}\big].
	\end{align*}
	On the one hand, using the monotone convergence theorem, one can obtain that
	\begin{align*}
	\lim\limits_{n\rightarrow \i}\frac{1}{\beta}\log\dbE\big[e^{\beta V^n}\big] = \frac{1}{\beta}\log\dbE\big[e^{\int_{0}^{T}\beta|\psi(s)|^2ds}\big].
	\end{align*}
	On the other hand, for $\dbQ \in \mathscr{Q}$, define 
	\begin{align*}
	F^n(\dbQ):= \dbE^{\dbQ}\big[V^n\big] -  \frac{1}{\beta}\dbE^\dbQ\Big[\log\frac{d\dbQ}{d\dbP} \Big], \q F(\dbQ):= \dbE^{\dbQ}\Big[\int_{0}^{T}|\psi(s)|^2ds\Big] -  \frac{1}{\beta}\dbE^\dbQ\Big[\log\frac{d\dbQ}{d\dbP} \Big].
	\end{align*}
	Then, by the monotone convergence theorem, we obtain 
	\begin{align}
	\lim\limits_{n\rightarrow \i} F^n(\dbQ) = F(\dbQ).
	\end{align}
	Next, we prove that
	\begin{align}\label{aa}
	\lim\limits_{n\rightarrow \i} \mathop{\sup}_{\dbQ\in \mathscr{Q}}F^n(\dbQ) = \mathop{\sup}_{\dbQ\in \mathscr{Q}}F(\dbQ).
	\end{align}
  For each $\dbQ\in \mathscr{Q}$, we have $F^n(\dbQ) \leq F(\dbQ)$. It then follows that 
		\begin{align}
		\lim\limits_{n\rightarrow \i} \mathop{\sup}_{\dbQ\in \mathscr{Q}}F^n(\dbQ) \leq \mathop{\sup}_{\dbQ\in \mathscr{Q}}F(\dbQ).
	\end{align}
	Furthermore, for $\dbQ \in \mathscr{Q}$ and $n \in \dbN$, one can obtain 
	\begin{align}\label{bs}
		 \mathop{\sup}_{\dbQ'\in \mathscr{Q}}F^n(\dbQ') \geq F^n(\dbQ).
	\end{align}
	Then, letting $n\rightarrow \i$ and taking $\mathop{\sup}_{\dbQ \in \mathscr{Q}}$ on both sides of \eqref{bs}, one can obtain that
		\begin{align}
		\lim\limits_{n\rightarrow \i} \mathop{\sup}_{\dbQ\in \mathscr{Q}}F^n(\dbQ) \geq \mathop{\sup}_{\dbQ\in \mathscr{Q}}F(\dbQ),
	\end{align}
	which implies \eqref{aa}.
	
	Consequently, the integrability in \eqref{u} and in Delarue and Lavigne \cite[(4)]{delarue2026robust} is equivalent in the case of $f= \frac{\gamma\beta}{2}|Z|^2$.

\end{remark}
The following assumptions are associated with the convex control domain.
\begin{assumption}\label{a2}
	\begin{itemize}
		\item [$\rm(i)$] 
		For $\mathbb P$-almost every $\omega$,
		$h(\cdot)$ is continuously differentiable.
		The derivative $h_x$ is essentially
		bounded in $\omega$, uniformly in $x$ and $h(0)$ is essentially bounded.
		\item [$\rm(ii)$]For $d\mathbb P\otimes dt$-almost every $(\omega,t)$, the mapping
		$(x, u)\mapsto b(t,x,u)$ and $ u\mapsto \sigma(t,u)$ are continuously differentiable.
		The derivatives $b_x$, $b_{u}$, and $\sigma_u$ are essentially
		bounded in $(\omega,t)$, uniformly in $(x,u)$.
		Moreover, there exists a deterministic constant $\tilde{L}_1>0$
		such that, for $d\mathbb P\otimes dt$-almost every
		$(\omega,t)$ and all $u\in U$,
		\begin{equation}\label{}
			|b(t,0,u)|+|\sigma(t,u)|\le \tilde{L}_1(1+|u|).
		\end{equation}
		\item [$\rm(iii)$] 
		For $d\mathbb P\otimes dt$-almost every $(\omega,t)$, the mapping
		$(x,y, z, u)\mapsto f(t,x,y,z,u)$ is continuously differentiable and the mapping $z \mapsto f(t, x, y, z, u)$ is convex. 
		Moreover, the second-order derivative $f_{zz}$ exists, and the
		derivatives $f_x$, $f_y$, $f_u$, and $f_{zz}$ are essentially bounded
		in $(\omega,t)$, uniformly in $(x,y,z,u)$.
		There exist deterministic positive constants
		$\tilde{L}_2,\tilde{L}_3,\tilde{L}_4$
		such that, for $d\mathbb P\otimes dt$-almost every
		$(\omega,t)$ and all
		$x\in\mathbb R^n$, $y\in\mathbb R$,
		$z,z'\in\mathbb R^d$, and $u,u_1,u_2\in U$,
		\begin{align}
			&|f(t, 0, 0, 0, u)| \leq \tilde{L}_2(1+|u|);\label{}\\
			&|f_z(t, x, y, z, u)| \leq \tilde{L}_3 (1+ |z|);\\
			&|f(t, x, y, z, u_1) - f(t, x, y, z, u_2)| \leq \tilde{L}_4(1 + |x| ).
		\end{align}
	\end{itemize}
\end{assumption}

Now we give the well-posedness result of system \eqref{system} under \autoref{a2}. 
\begin{lemma} \label{system-wellposedness1}
	Let \autoref{a2} hold. Then, there exists a constant $\beta$ such that for $u \in \mathcal{\widetilde{U}}^\beta_{ad}$, \eqref{system} admits a unique solution $ \big(X^u, Y^u, Z^u\big) $ satisfying
	\begin{align}
		\dbE\big[\mathop{\sup}_{0 \leq t \leq T} \exp\{\beta_1 |X^u(t)|\}\big]< \i, \q \dbE\big[\mathop{\sup}_{0 \leq t \leq T} \exp\{\beta_2 |Y^u(t)|\}\big] <\i,\q \dbE\Big[\Big(\int_{0}^{T} |Z^u(t)|^{2}dt\Big)^{\frac{\beta_2}{2}}\Big] < \i,\label{small1}
	\end{align}
	where $\beta_1$ and  $\beta_2$ are positive constants depending on $\beta$.
\end{lemma}
\begin{proof}
	On the one hand, let $N(t):= \int_{0}^{t}\sigma(s, u(s))dW(s)$ and $N^*(t):= \mathop{\sup}_{s\in [0, t]}|N(s)|$. 
	Using  Gronwall's inequality and H$\ddot{\rm o}$lder's inequality, we obtain
	\begin{align}\label{xup}
		&\mathop{\sup}_{s\in[0, T]}|X^u(s)|\leq e^{\|b_x\|_\i T}\Big(|x(0)|+\int_{0}^{T}|b\big(t, 0, u(t)\big)|dt  + N^*(T) \Big)\nonumber \\
		&\leq e^{\|b_x\|_\i T}\Big(\{\tilde{L}_1+|x(0)|\}(T + 1 )+ \tilde{L}_1\int_{0}^{T}|u(t)|dt  + N^*(T) \Big).
	\end{align}
	Furthermore, by Revuz and Yor \cite[Exercise (3.16) on page 153]{revuz2013continuous}, we obtain that 
	\begin{align}
		\dbP\Big( N^*(T) \geq r_1, \langle N\rangle_T \leq r_2  \Big) \leq 2e^{-\frac{r_1^2}{2r_2}}.
	\end{align}
	It then follows that
	\begin{align*}
		\dbP\Big( N^*(T) \geq r_1  \Big) \leq \dbP\Big( \langle N\rangle_T > r_2\Big) + 2e^{-\frac{r_1^2}{2r_2}} = \dbP\Big(e^{\beta'\langle N\rangle_T} > e^{\beta' r_2 }\Big) + 2e^{-\frac{r_1^2}{2r_2}} \leq e^{-\beta' r_2}\dbE\Big[e^{\beta' \langle N\rangle_T}\Big]+ 2e^{-\frac{r_1^2}{2r_2}}.
	\end{align*}
	Let
	\begin{align*}
		\beta' := \frac{\beta}{2\tilde{L}_1^2}, \q r_1:= \theta, \q r_2:= \frac{\theta}{\sqrt{2\beta}}.
	\end{align*}
	Then, one can deduce that
	\begin{align*}
		\dbP\Big( N^*(T) \geq \theta  \Big) \leq e^{\beta T} \exp\Big\{-\frac{\sqrt{\beta}\theta}{2\sqrt{2}\tilde{L}_1^2}\Big\} \dbE\Big[\exp\Big\{\beta 
		\int_{0}^{T}|u(t)|^2dt\Big\}\Big] +2\exp\Big\{-\frac{\sqrt{\beta}\theta}{\sqrt{2}}\Big\}
	\end{align*}
	In view of \eqref{xup} and using Young's inequality, we deduce that
	\begin{align*}
		&\dbE\Big[\exp\Big\{\mathop{\sup}_{s\in[0, T]}\beta_1|X^u(s)| \Big\}\Big] \\
		&\leq \exp\big\{\beta_1 e^{\|b_x\|_\i T} \{\tilde{L}_1+|x(0)|\} (T+1)\big\} \dbE\Big[\exp\big\{\beta_1 e^{\|b_x\|_\i T} N^*(T)\big\}\exp\Big\{\beta_1 e^{\|b_x\|_\i T}\tilde{L}_1 \int_{0}^{T}|u(t)|dt\Big\}\Big]\\
		&\leq K\dbE\Big[\exp\big\{2\beta_1 e^{\|b_x\|_\i T} N^*(T)\big\}\Big]^\frac{1}{2}.
	\end{align*}
	It then follows that
	\begin{align*}
		&\dbE\Big[\exp\big\{2\beta_1 e^{\|b_x\|_\i T}  N^*(T)\big\}\Big] =  1+\dbE\Big[\int_{0}^{\i}2\beta_1 e^{\|b_x\|_\i T} \exp\Big\{2\beta_1 e^{\|b_x\|_\i T}\theta\Big\} 1_{\{N^*(T)>\theta\}} d\theta\Big] \\
		&=  1+\int_{0}^{\i}2\beta_1 e^{\|b_x\|_\i T} \exp\Big\{2\beta_1 e^{\|b_x\|_\i T}\theta\Big\} \dbP\{N^*(T)>\theta\} d\theta\\
		&\leq 1+2\beta_1 e^{\|b_x\|_\i T}\bigg\{2 \int_{0}^{\i}\exp\Big\{\Big(2\beta_1 e^{\|b_x\|_\i T} - \frac{\sqrt{\beta}}{\sqrt{2}}\Big)\theta\Big\}  d\theta +e^{\beta T}\dbE\Big[\exp\Big\{\beta 
		\int_{0}^{T}|u(t)|^2dt\Big\}\Big]\\
		&\q \cd \int_{0}^{\i}\exp\Big\{\Big(2\beta_1 e^{\|b_x\|_\i T} - \frac{\sqrt{\beta}}{2\sqrt{2}\tilde{L}_1^2}\Big)\theta\Big\}  d\theta\bigg\}.
	\end{align*}
	Then, if $\beta_1 < \frac{\sqrt{\beta}}{2\sqrt{2}e^{\|b_x\|_\i T}(2+2\tilde{L}_1^2)}$, one can deduce that $\dbE\Big[\exp\Big\{\mathop{\sup}_{s\in[0, T]}\beta_1|X^u(s)| \Big\}\Big]<\i$.
	
	On the other hand, 
	we know that $\beta_1$ can be chosen sufficiently large if $\beta$ is sufficiently large. Therefore, by choosing such a $\beta_1$ and applying the result of Briand and Hu \cite {Briand2006quadratic} and Delbaen, Hu, and Richou \cite{delbaen2011uniqueness}, we obtain a unique solution $(Y^u,Z^u)$ to the BSDE satisfying \eqref{small1}.
\end{proof}
Let $\beta_0>0$ be sufficiently large so that,  for  $u \in \tilde{\mathcal{U}}^{\beta_0}_{ad}$, \eqref{system} admits a unique solution and the estimates required in the proof of \autoref{cor} hold. 
	Let $u^*(\cd) \in \mathcal{\widetilde{U}}^{\beta_0}_{ad}$ be an optimal control and $(X^*(\cd), Y^*(\cd), Z^*(\cd))$ be the corresponding state trajectories of \eqref{system}. For any $u(\cdot) \in \mathcal{\widetilde{U}}^{\beta_0}_{ad}$, $0<\e<1$, $M>0$,
define $u^{\e, M}(t):= u^*(t) + \e \big(u(t) - u^*(t)\big)1_{\{|u(t) - u^*(t)|\leq M, |u^*(t)|\leq M\}}$ and $\delta u^M (t):= \big(u(t) - u^*(t)\big)1_{\{|u(t) - u^*(t)|\leq M, |u^*(t)|\leq M\}}$.
One can check that $u^{\e, M}(t) \in U, t\in[0, T]$ and $u^{\e, M}(\cd) \in \tilde{\mathcal{U}}_{ad}^{\beta_0}$.
In addition, we let $(X^{\e, M}(\cdot), Y^{\e, M}(\cdot), Z^{\e, M}(\cdot))$ denote the state trajectories of \eqref{system} associated with $u^{\e, M} (\cdot)$. 
The derivatives $b_u, \sigma_u$, and $ f_u$ are evaluated along the optimal control and its corresponding state trajectory.
The variational equation for the SDE in \eqref{system} is
\begin{equation}\label{variational-x11}\left\{\begin{aligned}
		dX^M_1(t) = &\{b_{x}(t)X^M_1(t)+  b_u(t)\delta u^M(t)\}dt +  \sigma_u(t)\delta u^M(t) dW(t), \quad t\in [0,T],\\
		X^M_1(0) =&0.
	\end{aligned}\right.\end{equation}
Recall that $q^*$, $\dbQ^*$, and $W^{q^*}$ are defined in \eqref{qd}.
Since $\beta_1$ can be chosen sufficiently large if $\beta$ is sufficiently large, we have the following result.
\begin{proposition}\label{prop3.22}
	Let \autoref{a2} hold. Then, $\dbQ^*$ is a probability measure, $W^{q^*}$ is a Brownian motion under $\dbQ^*$, and $\dbE^{\dbQ^*}\Big[\int_{0}^{T}|q^*(t)|^2dt\Big] < \i$.
\end{proposition}
Set
$\hat{X}^{\e, M}:= X^{\e, M} - X^*$ and $\Delta X^{\e, M} := \frac{X^{\e, M}-X^*}{\e} - X^M_1$. We now give the estimates of them.

\begin{proposition}\label{state1}
	Let \autoref{a2} hold. Then, we have 
	\begin{align}
		\dbE^{\dbQ^*}\Big[\mathop{\sup}_{t\in [0, T]}|\hat{X}^{\e, M}(t)|^2\Big]  = O(\e^2), \q \dbE^{\dbQ^*}\Big[\mathop{\sup}_{s\in [0, T]}|X^M_1(s)|^2\Big]<\i, \q
		\lim\limits_{\e \rightarrow 0}\dbE^{\dbQ^*}\Big[\mathop{\sup}_{t\in [0, T]}|\Delta X^{\e, M}(t)|^2\Big]  = 0.
	\end{align}
\end{proposition}
\begin{proof}
	First,
	by the definition of $ \hat{X}^{\e, M}$, we have
	\begin{align*}
		\hat{X}^{\e, M}(t) = &\int_{0}^{t} \big\{\tilde{b}^{\e, M}_x(s)\hat{X}^{\e, M}(s) + \tilde{b}^{\e, M}_u(s)\delta u^M(s)\e + q^*(s)\tilde{\sigma}^{\e, M}_{u}(s)\delta u^{M}(s)\e\big\}ds \\
		&+ \int_{0}^{t}\tilde{\sigma}^{\e, M}_u(s)\delta u^M(s)\e dW^{q^*}(s),
	\end{align*}
	where $\tilde{b}_x^{\e, M}(s):= \int_{0}^{1}b_x\big(s, X^*(s)+\theta \hat{X}^{\e, M}(s), u^*(s)+\theta \delta u^M(s)\e\big)d\theta$, and $\tilde{b}_u^{\e, M}(s), \tilde{\sigma}_u^{\e, M}$ are defined similarly.
	It then follows from Gronwall's inequality, the Burkholder-Davis-Gundy's inequality and \autoref{prop3.22} that
	\begin{align*}
		\dbE^{\dbQ^*}\Big[\mathop{\sup}_{t\in [0, T]}|\hat{X}^{\e, M}(t)|^2\Big]  = O(\e^2).
	\end{align*}
	Second, by the definition of $X^M_1$, we have that 
	\begin{align*}
		X^M_1(t) = \int_{0}^{t} \big\{b_x(s)X^M_1(s) + b_u(s)\delta u^M(s)+ q^*(s)\sigma_u(s)\delta u^M(s)\big\}ds + \int_{0}^{t}\sigma_u(s)\delta u^M(s) dW^{q^*}(s).
	\end{align*}
	It then follows from Gronwall's inequality and the Burkholder-Davis-Gundy's inequality and \autoref{prop3.22} that 
	\begin{align*}
		\dbE^{\dbQ^*}\Big[\mathop{\sup}_{s\in [0, T]}|X^M_1(s)|^2\Big]<\i.
	\end{align*}
	By the fact that  $|\delta u^M|\leq M$
	and  the definition of $\Delta X^{\e, M}$, we have
	\begin{align*}
		\Delta X^{\e, M}(t) = &\int_{0}^{t} \Big\{\tilde{b}^{\e, M}_x(s)\Delta X^{\e, M}(s)+\big(\tilde{b}^{\e, M}_x(s) - b_x(s)\big)X^M_1(s) + \big(\tilde{b}^{\e, M}_u (s) - b_u(s)\big)\delta u^M(s) \\
		&
		+ q^*(s)\delta u^M(s)\big(\tilde{\sigma}^{\e, M}_u(s) - \sigma_u(s)\big)\Big\}ds + \int_{0}^{t} \big(\tilde{\sigma}^{\e, M}_u(s) - \sigma_u(s)\big)\delta u^M(s)dW^{q^*}(s).
	\end{align*}
	It then follows from Gronwall's inequality, the Burkholder-Davis-Gundy's inequality that
	\begin{align*}
		&\limsup\limits_{\e \rightarrow 0}\dbE^{\dbQ^*}\Big[\mathop{\sup}_{t\in [0, T]}|\Delta X^{\e, M}(t)|^2\Big] \\
		&
		\leq K\limsup\limits_{\e \rightarrow 0}\dbE^{\dbQ^*}\Big[\Big(\int_{0}^{T}\Big\{\big|\tilde{b}^{\e, M}_x(s) - b_x(s)\big||X^M_1(s)| + \big|\tilde{b}^{\e, M}_u(s) - b_u(s)\big||\delta u^M(s) |\\
		&\q\q\q\q\q\q\q+ |q^*(s)|\big| \tilde{\sigma}^{\e, M}_u(s) - \sigma_u(s)\big||\delta u^M(s)|  \Big\}ds\Big)^2 
		+ \int_{0}^{T}|\big( \tilde{\sigma}^{\e, M}_u(s) - \sigma_u(s)\big)\delta u^M(s) |^2ds\Big].
	\end{align*}
	Note the fact that 
	\begin{align*}
		&\big|\tilde{b}^{\e, M}_x(s) - b_x(s)\big||X^M_1(s)| + \big|\tilde{b}^{\e, M}_u(s) - b_u(s)\big||\delta u^M(s) |+ |q^*(s)|\big| \tilde{\sigma}^{\e, M}_u(s) - \sigma_u(s)\big||\delta u^M(s)|  \\
		&
		\leq 2(\|b_x\|_\i + \|b_u\|_\i+\|\sigma_u\|_\i) \Big(|X^M_1(s)| + |\delta u^M(s)| + |q^*(s)\delta u^M(s)|\Big),\\
		&|\big(\tilde{\sigma}^{\e, M}_u(s) - \sigma_u(s)\big)\delta u^M(s) |^2 \leq 4\|\sigma_u\|^2_\i|\delta u^M(s) |^2,
	\end{align*}
	and 
	\begin{align*}
		\dbE^{\dbQ^*}\Big[\int_{0}^{T}\Big\{|X^M_1(s)|^2 + |\delta u^M(s)|^2 + |q^*(s)\delta u^M(s)|^2 \Big\}ds\Big] <\i, \q \dbE^{\dbQ^*}\Big[\int_{0}^{T}|\delta u^M(s) |^2ds\Big] <\i.
	\end{align*}
	It then follows from the dominated convergence theorem that 
	\begin{align*}
	\limsup\limits_{\e \rightarrow 0}\dbE^{\dbQ^*}\Big[\mathop{\sup}_{t\in [0, T]}|\Delta X^{\e, M}(t)|^2\Big] = 0.
	\end{align*}
	This completes the proof.
\end{proof}
Set $\big(\hat{Y}^{\e, M}(t), \hat{Z}^{\e, M}(t)\big) := \big(Y^{\e, M}(t) - Y^*(t), Z^{\e, M}(t) - Z^*(t)\big)$. 
Applying an argument similar to \autoref{importan-} and using \autoref{prop3.22}, we also have the following result.
\begin{proposition}\label{importan--}
	Let \autoref{a2} hold. Then, we have that 
	\begin{align}
		\dbE^{\dbQ^*}\Big[\mathop{\sup}_{t\in [0, T]} |\hat{Y}^{\e, M}(t)|^2 \Big] = O(\e^2) \q\hb{and} \q \dbE^{\dbQ^*}\Big[\int_{0}^{T}|\hat{Z}^{\e, M}(t)|^2dt\Big] = O(\e^2).
	\end{align}
\end{proposition}

Let  the variational equation for $\big(\hat{Y}^{\e, M}, \hat{Z}^{\e, M}\big)$ be the following:
\begin{align*}
	Y_1^M(t) = &h_x(T)X_1^M(T) - \int_{t}^{T} \Big\{f_x(s)X_1^M(s)+ f_y(s)Y_1^M(s)+f_z(s) Z_1^M(s)+f_u(s)\delta u^M(s)\Big\}ds\\
	& + \int_{t}^{T}Z_1^M(s)dW(s), \q t\in[0, T].
\end{align*}
Define 
\begin{align}
	\big(\Delta Y^{\e, M}, \Delta Z^{\e, M}\big) := \Big(\frac{\hat{Y}^{\e, M}}{\e} - Y_1^M, \frac{\hat{Z}^{\e, M}}{\e} - Z_1^M\Big).
\end{align}
Then, we have the following result.
\begin{lemma}\label{iii}
	Let \autoref{a2} hold. Then, we have that
	\begin{align}\label{la}
		\lim\limits_{\e \rightarrow 0} \Delta Y^{\e, M}(0) = 0.
	\end{align}
\end{lemma}
\begin{proof}
	By the definition of $\big(\Delta Y^{\e, M}, \Delta Z^{\e, M}\big)$, one can obtain that
	\begin{align*}
		&\Delta Y^{\e, M}(t) = \frac{h\big(X^{\e, M}(T)\big) - h(T) }{\e}- h_x(T) X_1^M(T) - \int_{t}^{T} \Big\{\frac{1}{\e} \Big(f\big(s, X^{\e, M}(s), Y^{\e, M}(s), Z^{\e, M}(s), u^{\e, M}(s)\big) \\
		&- f(s)\Big) - f_x(s) X_1^M(s) - f_y(s)Y_1^M(s) - f_z(s) Z_1^M(s) - f_u(s)\delta u^M(s)\Big\}ds + \int_{t}^{T}\Delta Z^{\e, M}(s) dW(s).
	\end{align*}
	Applying Taylor's expansion, one can deduce that
	\begin{align*}
		&\Delta Y^{\e, M}(t) = \Pi^{\e, M}(T) - \int_{t}^{T} \Big\{\frac{1}{\e}\big\{f_x^\e(s) - f_x(s)\big\}\hat{X}^{\e, M}(s) + \frac{1}{\e}\big\{f_y^\e(s) - f_y(s)\big\}\hat{Y}^{\e, M}(s) \\
		&+ \frac{1}{\e}\big\{f_z^\e(s) - f_z(s)\big\}\hat{Z}^{\e, M}(s) + \big\{f_u^\e(s) - f_u(s)\big\}\delta u^M(s)  + f_x(s) \Delta X^{\e, M}(s) + f_y(s)\Delta Y^{\e, M}(s) 
		\Big\}ds \\
		&+ \int_{t}^{T}\Delta Z^{\e, M}(s) dW^{q^*}(s),
	\end{align*}
	where
	\begin{align}
		&\Pi^{\e, M} (T):=\frac{1}{\e}\big\{h_x^\e(T) - h_x(T)\big\}\hat{X}^{\e, M}(T) + h_x(T) \Delta X^{\e, M}(T),\nonumber\\
		& h_x^\e(T):= \int_{0}^{1}h_x\big(X^*(T)+\theta \hat{X}^{\e, M}(T)\big)d\theta,\nonumber\\
		&f^\e_x(s): = \int_{0}^{1}f_x\big(s, X^*(s)+\theta \hat{X}^{\e, M}(s), Y^*(s)+\theta \hat{Y}^{\e, M}(s),Z^*(s)+\theta \hat{Z}^{\e, M}(s), u^*(s) + \theta \delta u^M(s)\e\big)d\theta,
	\end{align}
	and $f_y^\e(s), f_z^\e(s), f_u^\e(s)$ are defined similarly.
	Recall that $\Gamma$ is defined in \eqref{ga}.
	Applying It$\hat{\rm o}$'s formula to $\big\{\Gamma (t)\Delta Y^{\e, M}(t)\big\}_{t\in [0, T]}$, one can obtain that
	\begin{align*}
		\Delta Y^{\e, M}(0) = \dbE^{\dbQ^*}
		\Big[&\Gamma(T)\Pi^{\e, M}(T) - \int_{0}^{T}\Gamma(s)\Big\{\frac{1}{\e}\big\{f_x^\e(s) - f_x(s)\big\}\hat{X}^{\e, M}(s) + \frac{1}{\e}\big\{f_y^\e(s) - f_y(s)\big\}\hat{Y}^{\e, M}(s)\\
		&+ \frac{1}{\e}\big\{f_z^\e(s) - f_z(s)\big\}\hat{Z}^{\e, M}(s) + \big\{f_u^\e(s) - f_u(s)\big\}\delta u^M(s) + f_x(s) \Delta X^{\e, M}(s) \Big\}ds\Big].
	\end{align*}
	It then follows from H$\ddot{\rm o}$lder's inequality, \autoref{state1}, \autoref{importan--}, and   Dominated convergence theorem that \eqref{la} holds.

\end{proof}

Let the adjoint equation be the following:
\begin{align}
	P(t) = \Gamma(T)h_x(T) - \int_{t}^{T} \Big\{ -b_x(s)P(s) + \Gamma(s)f_x(s) \Big\}ds + \int_{t}^{T}Q(s)dW^{q^*}(s).
\end{align}
Recall that 
\begin{align*}
	Y_1^M(t) = &h_x(T)X_1^M(T) - \int_{t}^{T} \Big\{f_x(s)X_1^M(s)+ f_y(s)Y_1^M(s)+f_u(s)\delta u^M(s)\Big\}ds\\
	& + \int_{t}^{T}Z_1^M(s)dW^{q^*}(s), \q t\in[0, T].
\end{align*}
Then, we have the following result:
\begin{lemma}
	Let \autoref{a2} hold. Then, we have that
	\begin{align}\label{last1}
		Y_1^M(0) = \dbE^{\dbQ^*}\Big[\int_{0}^{T} G(s) \delta u^M(s)ds\Big],
	\end{align}
	where
	\begin{align}
		G(s) = P(s) \big[b_u(s)+f_z(s)\sigma_u(s)\big]-Q(s)\sigma_u(s) -  \Gamma(s)f_u(s).
	\end{align}
\end{lemma}
\begin{proof}
	Applying It$\hat{\rm o}$'s formula to $\{\Gamma(s)Y^M_1(s)\}_{s\in [0, T]}$, we have that
	\begin{align*}
		Y^M_1(0) = \dbE^{\dbQ^*}\Big[\Gamma(T)h_x(T)X_1^M(T) - \int_{0}^{T}\Gamma(s)\big\{f_x(s)X_1^M(s)+f_u(s)\delta u^M(s)\big\}ds\Big]
	\end{align*}
	Note the fact that
	\begin{align}
		X^M_1(t) = \int_{0}^{t}\{b_{x}(s)X^M_1(s)+ f_z(s)\sigma_u(s)\delta u^M(s) +  b_u(s)\delta u^M(s)\}ds +  \int_{0}^{t} \sigma_u(s)\delta u^M(s) dW^{q^*}(s), \quad t\in [0,T].
	\end{align}
	Then, applying It$\hat{\rm o}$'s formula to $\big\{P(s)X^M_1(s)\big\}_{s\in[0, T]}$, one can obtain \eqref{last1}.
	
\end{proof}

Define
\begin{align}
	\big(\tilde{P}(s), \tilde{Q}(s)\big) := \big(\Gamma(s)^{-1}P(s), \Gamma(s)^{-1}Q(s)\big),
\end{align}
one can observe that
\begin{align}\label{adcon}
	\tilde{P}(t) = h_x(T) - \int_{t}^{T}\Big\{\big[-b_x(s) + f_y(s)\big]\tilde{P}(s) + f_x(s)\Big\}ds + \int_{t}^{T}\tilde{Q}(s)dW^{q^*}(s)
\end{align}
and $\big(\tilde{P}, \tilde{Q}\big) \in S_{\dbF}^\i(\dbQ^*) \times L_{\dbF}^\beta(\dbQ^*), \beta>1$.
Based on the above preparations, we have the following result.
\begin{corollary}\label{cor}
	Suppose that \autoref{a2}  holds. Let $u^*(\cd) \in \mathcal{\widetilde{U}}^{\beta_0}_{ad}$ be an optimal control and $(X^*(\cd), Y^*(\cd), Z^*(\cd))$ be the corresponding state trajectories of \eqref{system}. Then the following stochastic maximum principle holds:
	\begin{align}\label{smp}
		\Big\{\tilde{P}(s) \big[b_u(s)+f_z(s)\sigma_u(s)\big]-\tilde{Q}(s)\sigma_u(s) -  f_u(s) \Big\} \Big(u - u^*(s)\Big) \geq 0, \q \forall u\in U, \q a.e., \q a.s.,
	\end{align}
	where $\big(\tilde{P}, \tilde{Q}\big)$ is the unique solution of \eqref{adcon}.
\end{corollary}
\begin{proof}
	By \autoref{iii}, one can deduce that
	\begin{align*}
		0\leq \lim\limits_{\e \rightarrow 0}\frac{J\big(u^{\e, M}(\cd)\big) - J\big(u^*(\cd)\big)}{\e} = Y_1^{ M}(0).
	\end{align*}
	It then follows that 
	\begin{align}
		Y_1^M(0) = \dbE^{\dbQ^*}\Big[\int_{0}^{T} G(s) \delta u^M(s)ds\Big]\geq 0.
	\end{align}
	By a standard argument, one can deduce that for $u\in U$,
	\begin{align*}
		\Big\{\tilde{P}(s) \big[b_u(s)+f_z(s)\sigma_u(s)\big]-\tilde{Q}(s)\sigma_u(s) -  f_u(s) \Big\} \Big(u - u^*(s)\Big)1_{\{|u - u^*(s)|\leq M, |u^*(s)|\leq M\}} \geq 0,  \q a.e., \q a.s..
	\end{align*}
	Finally, by letting $M \rightarrow \i$, we obtain \eqref{smp}.
	
\end{proof}

\begin{remark}
	In the nonconvex control case, the second-order expansion requires the estimates in \autoref{q*}. The truncation argument used in the convex case does not directly yield these estimates. Therefore, we impose a boundedness assumption on $\sigma$ in our analysis of the nonconvex case.
\end{remark}

\noindent{\bf Acknowledgement.}
The authors would like to thank Dr. Pierre Lavigne for his active participation in the discussions.

\noindent{\bf Declarations.} 
AI was utilized during the drafting process solely to enhance the linguistic quality, correct grammar, and provide a preliminary check of the mathematical consistency. All core concepts and mathematical proofs remain the original work of the authors, who retain sole accountability for the accuracy of the paper.


\end{document}